\documentclass{amsart}

\usepackage{amsfonts, amsmath, amscd}
\usepackage[psamsfonts]{amssymb}
\usepackage{amssymb}
\usepackage{mathrsfs}
\usepackage{pb-diagram}
\usepackage[usenames]{color}
\usepackage{hyperref}
\usepackage[all]{xy}
\usepackage{graphicx}
\usepackage{array}
\newcolumntype{P}[1]{>{\centering\arraybackslash}p{#1}}
\newtheorem{deff}{Definition}[section]
\newtheorem{lemma}[deff]{Lemma}
\newtheorem{thm}[deff]{Theorem}
\newtheorem{corollary}[deff]{Corollary}

\newtheorem{prop}[deff]{Proposition}

\newtheorem{em-example}[deff]{Example}
\newtheorem{em-def}[deff]{Definition}        
\newtheorem{em-remark}[deff]{Remark}         
\newtheorem{em-question}[deff]{Question}

\newtheorem{rem}[deff]{Remark}
\newtheorem{problem}[deff]{Problem}

\newcommand{\R}{\mathbb{R}}

\def\N{\mathbb N}
\def\R{\mathbb R}
\def\T{\mathbb T}

\def\Z{\mathbb Z}

\DeclareMathOperator*{\supp}{\rm supp}

\DeclareMathOperator{\res}{\restriction}

\def\cont{\mathfrak c}

\def\hull#1{\langle#1\rangle}

\def\rank{\mathop{\hbox{{\rm rank}}}\nolimits}

\begin{document}

\title[The Separable Quotient Problem for Topological Groups]{The Separable Quotient Problem for Topological Groups}
\author{Arkady G. Leiderman, Sidney A. Morris, and Mikhail G. Tkachenko}
\address{Department of Mathematics, Ben-Gurion University of the Negev, Beer Sheva, P.O.B. 653, Israel}
\email{arkady@math.bgu.ac.il}
\address{Faculty of Science, Federation University Australia, P.O.B. 663, Ballarat, Victoria, 3353,  Australia
\newline
Department of Mathematics and Statistics, La~Trobe University, Melbourne, Victoria, 3086, Australia}
\email{morris.sidney@gmail.com}
\address{Departamento de Matem\'aticas, Universidad Aut\'onoma Metropolitana, Av. San Rafael Atlixco 186, Col. Vicentina, Del. Iztapalapa, C.P. 09340, M\'exico, D.F., Mexico}
\email{mich@xanum.uam.mx}

\date{July 29, 2017}

\begin{abstract}
The famous Banach-Mazur problem,  which asks if  every infinite-dimensional Banach space has  an infinite-dimensional separable quotient Banach space, has remained  unsolved for 85 years, though it has been answered in the affirmative for reflexive Banach spaces and even Banach spaces which are duals. The analogous problem for locally convex spaces has been answered in the negative, but has been shown to be true for large classes of locally convex spaces including all non-normable Fr\'echet spaces. In this paper the analogous problem for topological groups is investigated. Indeed there are four natural analogues: Does every non-totally disconnected topological group have a  separable quotient group which is (i) non-trivial; (ii) infinite; (iii) metrizable; (iv) infinite metrizable. All four questions are answered here in the negative. However, positive answers are proved for important classes of topological groups including (a) all compact groups; (b) all locally compact abelian groups; (c) all $\sigma$-compact locally compact groups; (d) all abelian pro-Lie groups; (e) all $\sigma$-compact pro-Lie groups; (f) all pseudocompact groups. Negative answers are proved for precompact groups. 
\end{abstract}

\maketitle

\section{Introduction} 

\noindent 
It is natural to attempt to describe all objects of a certain kind in terms of basic building blocks of that kind. For example one may try to describe general Banach spaces in terms of separable Banach spaces. Recall that a  topological space is said to be \emph{separable} if it has a countable dense subset.

To put our investigation into context, we begin with a famous unsolved problem in Banach space theory. 
The Separable Quotient Problem for Banach Spaces has its roots in the 1930s and, according to a private communication from W\l adys\l aw  Orlicz to Jerzy K{\c{a}}kol, is due to Stefan Banach and 
Stanis\l aw Mazur.

\begin{problem}\label{SQProblem}{\textbf{Separable Quotient Problem for Banach Spaces.}}\quad Does every infinite-dimensional Banach space have a quotient Banach space which is  {separable} and infinite-dimensional?
\end{problem}

\smallskip
\noindent The related Quotient Schauder Basis Problem for Banach Spaces  is due to Aleksander (Olek) Pe\l czy\'nski \cite{Pelczynski}.

\begin{problem}{\textbf{Quotient Schauder Basis Problem for Banach Spaces.}\label{SQSchauder} \  Does every infinite-dimensional Banach space have a quotient Banach space which is infinite-dimensional and has  a Schauder basis?}
\end{problem}

\smallskip
\noindent Of course any Banach space with a Schauder basis is separable. Mazur's Problem 153 in The Scottish  Book \cite{Mauldin}, for which the prize was a live goose, was answered by Per Enflo \cite{Enflo} in 1973. Enflo proved that there exist separable Banach spaces which do not have a Schauder basis (and hence lack the approximation property).

\smallskip
However William Johnson and Haskell Rosenthal proved the following result:

\begin{thm} \label{JohnsonRosenthal}\textbf{\cite{JohnsonRosenthal}}
Every separable infinite-dimensional Banach space has a quotient infinite-dimensional Banach space with a Schauder basis. 
\end{thm}

\begin{corollary}\label{SchauderBasis}
The Quotient Schauder Basis Problem for Banach Spaces \ref{SQSchauder} and the Separable Quotient Problem for Banach Spaces \ref{SQProblem} are equivalent.
\end{corollary}

Steve Saxon and Albert Wilansky proved some equivalent versions of the Separable Quotient Problem for Banach Spaces.

\begin{thm}\label{SaxonWilansky}\cite{SaxonWilansky}
The following are equivalent for an infinite-dimensional Banach space $B$:\begin{itemize}
\item[(i)] $B$ has a quotient Banach space which is separable and infinite-dimensional;
\item[(ii)] $B$ has a dense subspace which is not barrelled;
\item[(iii)] $B$ has a dense subspace $E$ which is the union of a strictly increasing sequence of closed linear subspaces.
\end{itemize}
\end{thm}

Some extensions of this result to topological vector spaces are {obtained} in \cite{KakolSliwa}. 
We have then that Problem \ref{SQProblem} is equivalent to each of Problem \ref{UnionDenseLinearSubspaces} and Problem \ref{barrelled}

\begin{problem}\label{UnionDenseLinearSubspaces} Does every infinite-dimensional Banach space have a dense subspace $E$ which is the union of a strictly increasing sequence of closed linear subspaces?
\end{problem}

\begin{problem}\label{barrelled} Does every infinite-dimensional Banach space have a dense infinite-dimensional subspace which is not barrelled? 
\end{problem}

As a corollary of Theorem~\ref{SaxonWilansky}  (see Corollary~3.5 of \cite{Mujica}) one {obtains} the result first proved by Dan Amir and Joram Lindenstrauss. 

\begin{corollary}\label{AmirLindenstrauss} \cite{AmirLindenstrauss}
 Every infinite-dimensional weakly compactly generated (WCG) Banach space has a separable infinite-dimensional quotient Banach space. 
 \end{corollary}

 As reflexive Banach spaces (and separable Banach spaces) are WCG, one {obtains}:
 
 \begin{corollary}\label{reflexiveBanach} 
\cite{Pelczynski}
Every infinite-dimensional reflexive Banach space has a separable infinite-dimensional quotient Banach space. 
 \end{corollary}
 
 From Corollary~\ref{reflexiveBanach} one easily {obtains}:
 
 \begin{corollary}\label{separablesubspaceofdual} Let $B$ be a Banach space such that the dual Banach space $B^*$ has an infinite-dimensional reflexive subspace $E$, then $B$ has a quotient Banach space isomorphic to $E^*$. So $B$ has an infinite-dimensional separable quotient Banach space.
 \end{corollary}
 
 Spiros Argyros, Pandelis Dodos and Vassilis Kanellopoulos  in 2008 generalized  Corollary~\ref{reflexiveBanach}.
 
 \begin{thm}\cite{Argyros1}
\label{Argyros1} If $B$  is the Banach dual of any infinite-dimensional Banach space
 then $B$ has a separable infinite-dimensional quotient Banach space. 
 \end{thm}
 
  In the literature many special cases of the Separable Quotient Problem for Banach Spaces have been proved, however the general problem remains unsolved.

\smallskip
Turning to locally convex spaces one can state the analogous problem. Throughout this paper all locally convex spaces will be assumed to be Hausdorff.

\begin{problem}\label{SQlcs}\textbf{Separable Quotient Problem for Locally Convex Spaces.}\quad Does every infinite-dimensional locally convex space have a quotient locally convex  space which is  \emph{separable} and infinite-dimensional?
\end{problem}

\smallskip
 This question was answered in the positive for a wide class of locally convex spaces by M. Eidelheit \cite{Eidelheit}. (See also Chapter~6, Section~31.4 of \cite{Koethe}.)

\begin{thm}\label{SQFrechet}\cite{Eidelheit} Every infinite-dimensional Fr\'echet space (= locally convex space with its topology  determined by a complete translation invariant metric) which is non-normable has the separable metrizable topological vector space $\R^{\omega}$ as a quotient space.
\end{thm}

\begin{deff}\cite{Robertson} 
{\rm A topological vector space is said to be \emph{properly separable} if it has a proper dense vector subspace of countably infinite (Hamel) dimension.}
\end{deff}

Wendy Robertson \cite{Robertson} observed that a Fr\'echet space is properly separable if and only if it is separable and that a metrizable barrelled locally convex space is properly separable if and only if it is separable. The following result  generalizes Theorem~\ref{SQFrechet}. For further generalizations along this line, see \cite{SaxonN}.

\begin{thm}\label{strict(LF)}\cite{Robertson} Every strict inductive limit of a strictly increasing sequence $(E_m)$ of  Fr\'echet spaces with at least one $E_m$ non-normable has a properly separable quotient locally convex space.
\end{thm}

Jerzy K{\c{a}}kol, Steve Saxon and Aaron Todd \cite{KakolSaxonTodd} answered Problem~\ref{SQlcs} in the negative.

\begin{thm}\label{barrelledlcs}\cite{KakolSaxonTodd} There exist infinite-dimensional barrelled locally convex spaces which do not have any infinite-dimensional separable quotient locally convex spaces.
\end{thm}

However the following results go in the positive direction. 

\begin{thm}\label{C(X)}\cite{KakolSaxonTodd} Let $X$ be any infinite Tychonoff space and $C_c(X)$ the linear space of all real-valued continuous functions on $X$ endowed with the compact-open topology. If $C_c(X)$ is barrelled, then it has a quotient locally convex space which is infinite-dimensional and separable. 
\end{thm}

\begin{thm}\label{C(X)duals}\cite{KakolSaxonTodd} Let $X$ be any infinite Tychonoff space. Both the strong and weak duals of $C_c(X)$ have a quotient locally convex space which is infinite-dimensional and separable. \end{thm}

\medskip
Once again we note that there are many partial positive solutions in the literature to Problem \ref{SQlcs} (see also \cite{KakolSaxon}) which are out of the scope of these introductory remarks. 

\medskip
Now we turn to the specific topic of this paper, the Separable Quotient Problem(s) for Topological Groups. 
We shall in fact state ten quite natural problems.

\smallskip
\noindent {\bf Terminology and Basic Facts.} \quad  A topological space $X$ is said to be \emph{hereditarily separable} if $X$ and every subspace of $X$ is separable. A topological space is said to be \emph{second countable} if its topology has a countable base.
A topological space $X$ is said to have  a \emph{countable network} if there exists a countable family ${\mathcal B}$ of  (not necessarily open) subsets
such that each open set of $X$ is a union of members of $\mathcal B$. 
\begin{itemize}
\item[(i)] Any space with a countable network is hereditarily  separable;  
\item[(ii)] a metrizable space is separable if and only if it is second countable;
\item[(iii)] any continuous image of a separable space is separable; 
\item[(iv)] countable networks are preserved by continuous images; 
\item [(v)] a quotient image of a second countable space need not be second countable; 
\item[(vi)]  a closed subgroup of a separable topological group is not necessarily separable.
However, the class of topological groups with the property that every closed subgroup is separable
is quite large and natural, and in particular it
 includes all separable locally compact groups. This class of topological groups
has been thoroughly investigated  in \cite{LMT(TAMS_2017)}, \cite{LT(TA_2017)}. 
\end{itemize}
An  abstract group is called \emph{simple} if it has no proper non-trivial normal subgroup, where a group is called \emph{non-trivial} if it has at least two elements.  A topological group $G$ is said to be  \emph{topologically simple} if it has no proper non-trivial closed normal subgroup. As the connected component of the identity of any topological group is a closed normal subgroup, every topologically simple group is either  totally disconnected or connected. 
\newline Throughout this paper  all topological groups are assumed to be Hausdorff. The circle group 
with the usual multiplication and compact topology inherited from the complex plane is denoted 
by $\T$.  The product of elements $x,y\in\T$ will be denoted by $x\cdot y$, while we will use 
additive notation for all abelian groups distinct from $\T$. A \emph{character} of a group $H$ is 
a homomorphism of $H$ to $\T$. The additive topological group of all real numbers with the 
euclidean topology  is denoted by $\R$. The cardinality of the continuum is denoted by $\cont$, 
so $\cont=2^\omega$. By Lie group we mean a real finite-dimensional Lie group; the $0$-dimensional Lie groups are discrete. \qed

\smallskip
We begin with an example which says that we must exclude some discrete groups and indeed some totally disconnected groups.

 \begin{em-example}\label{Discrete} 
 {\rm Let $\aleph$ be any uncountable cardinal number and $S$ a set of cardinality $\aleph$. Let $A_\aleph$ be the \emph{finitary alternating group on the set $S$}; that is, $A_\aleph$ is the group of all even permutations of the set $S$ which fix all but a finite number of members of $S$. That  $A_\aleph$ is an uncountable simple group follows easily from \cite[3.2.4]{Robinson}.  So if $A_\aleph$ is given the discrete topology, it is an infinite discrete topologically simple group and so  does not have a (proper) quotient group which is a non-trivial  topological group.  
 
George A.~Willis \cite[\S 3]{Willis} (see also \cite{Capracepreprint})
 extended these examples of discrete topologically simple groups to produce, for each infinite cardinal number $\aleph$, a non-discrete  totally disconnected locally compact topologically simple group of cardinality (and weight) $\aleph$. 
 
 By contrast, Theorem~\ref{SQLCATheorem} shows that every infinite discrete abelian group $G$ does indeed have a quotient group which is a countably infinite (discrete) group. (Indeed  $G$ has a quotient group  of  cardinality $\kappa$, for each infinite $\kappa$ less than the cardinality of $G$.)} \qed
 \end{em-example}

\begin{problem}\label{SQTopGps} \textbf{Separable Quotient Problem for Topological Groups.} Does every   non-totally disconnected topological group have a quotient group which is a non-trivial separable topological group? 
\end{problem}

\begin{problem}\label{SQInfiniteTopGps} \quad \textbf{Separable Infinite Quotient Problem for Topological Groups.} Does every  non-totally disconnected  topological group have a  quotient group which is an infinite separable topological group? \end{problem}

\begin{problem}\label{SQMetrizableTopGps} \textbf{Separable Metrizable Quotient Problem for Topological Groups.} Does every non-totally disconnected topological group have a quotient group which is a non-trivial separable metrizable topological group? 
\end{problem}

\begin{problem}\label{SQInfiniteMetrizableTopGps} \textbf{Separable Infinite Metrizable Quotient Problem for Topological Groups.} Does every  non-totally disconnected topological group have 
a quotient group which is an  infinite separable metrizable topological group? 
\end{problem}

 Regarding Problem \ref{SQInfiniteMetrizableTopGps}, one might reasonably ask: If the topological group $G$ has a quotient group which is infinite and separable, does $G$ necessarily have a quotient group which is infinite, separable and metrizable? This question is answered negatively in Proposition \ref{Ex:Qs}.

\smallskip
Theorem~\ref{SaxonWilansky} suggests the following problem for topological groups.

\smallskip
We shall define a topological group to be a \emph{$\boldsymbol{G_\sigma}$-group} if it has a dense subgroup $H$ which is the union of a strictly increasing sequence of closed topological subgroups.

\begin{problem}\label{SQGsigma}\textbf{Separable Quotient Problem for $\boldsymbol{G_\sigma}$-Groups.}  Does every non-discrete $G_\sigma$-group have a separable quotient group which is (i) non-trivial; (ii) infinite; (iii) metrizable; (iv) infinite metrizable? 
\end{problem}

Corollary \ref{reflexiveBanach} suggests the following problem for \emph{reflexive} topological groups; that is, topological groups for which the natural map of the topological group into the Pontryagin dual of its Pontryagin dual is an isomorphism of topological groups. 

\begin{problem}\label{SQReflexiveTopGp} \textbf{Separable Quotient Problem for Reflexive Topological\\ Groups.} Does every infinite reflexive abelian topological group, $G$, have a  separable quotient group which is (i) non-trivial; (ii) infinite; (iii) metrizable; (iv) infinite metrizable?
\end{problem}

A special case of Problem~\ref{SQReflexiveTopGp} is:

\begin{problem}\label{SQLCA}  \textbf{Separable Quotient Problem for Locally Compact Abelian Groups.} Does every infinite locally compact abelian group, $G$, have a  separable quotient group which is (i) non-trivial; (ii) infinite; (iii) metrizable; (iv) infinite metrizable?
\end{problem}

As another special case of Problem \ref{SQReflexiveTopGp} one might be tempted to ask whether every topological group, $G$, which is the underlying topological group of an infinite-dimensional Banach space has  a separable quotient group which is (i) non-trivial; (ii) infinite; (iii) metrizable; (iv) infinite metrizable? But a positive answer to all these questions follows immediately from the fact that $\R^n$, for every $n\in \N$, is a a quotient locally convex space space of every infinite-dimensional locally convex space. Noting that
according to Anderson-Kadec theorem 
 each infinite-dimensional separable Banach space (indeed each infinite-dimensional separable Fr\'echet space) is homeomorphic to $\R^\omega$ (\cite[Chapter VI, Theorem 5.2]{BessagaPelczynski}), the next question is pertinent. Of course a positive answer to Problem \ref{SQProblem} would yield a positive answer to Problem \ref{SQBanachTopGp}.

\begin{problem}\label{SQBanachTopGp}\qquad \textbf{Separable Quotient Problem for Banach Topological\newline Groups.} Does every topological group, which is the underlying topological group of an infinite-dimensional  Banach space have a separable quotient group which is homeomorphic to   $\R^\omega$?
\end{problem}

The non-abelian version of Problem \ref{SQLCA} is:

\begin{problem}\label{SQLC}  \textbf{Separable Quotient Problem for Locally Compact  Groups.} Does every non-totally disconnected locally compact  group, $G$, have a  separable quotient group which is (i) non-trivial; (ii) infinite; (iii) metrizable; (iv) infinite metrizable?
\end{problem}

As a special case of Problem~\ref{SQLC} we have:

\begin{problem}\label{SQCompact}  \textbf{Separable Quotient Problem for Compact  Groups.} Does every  infinite compact  group, $G$, have a  separable quotient group which is  (i) non-trivial; (ii) infinite; (iii) metrizable; (iv) infinite metrizable?
\end{problem}

We shall address these problems in subsequent sections.


\section{Locally Compact Groups and Pro-Lie Groups} 

\noindent In this section we provide a positive answer to each of Problem~\ref{SQCompact} (i), (ii), (iii), and (iv) and Problem~\ref{SQLCA} (i), (ii), (iii), and (iv), and a partial answer to Problem~\ref{SQLC}. We prove satisfying results for  pro-Lie groups. We also prove stronger structural  results for compact abelian groups, connected compact groups and totally disconnected compact groups.

\begin{thm}\label{compactabelian} Every non-separable compact abelian group $G$ has a quotient group $Q$ which is  a countably infinite product of non-trivial compact  finite-dimensional Lie groups. The quotient group, $Q$, is therefore an infinite separable metrizable group. 
\end{thm}

\begin{proof} \quad \rm  Case 1.  $G$ is totally disconnected. By Corollary~8.8\,(iii) of \cite{COMPBOOK}, $G$ is a direct product of $p$-groups. By \cite{COMPBOOK} E8.4, $G$ is therefore a direct product of cyclic groups of bounded order. As $G$ is not separable, this cannot be a finite product.  Consequently $G$ has a quotient group which is a countably infinite direct product of cyclic groups of bounded order and so the theorem is true in Case~1.

Case 2. $m=\rank \widehat G$ is infinite, where $\widehat G$ denotes the Pontryagin dual of $G$. 
By Proposition~8.15 of \cite{COMPBOOK}, $G$ contains a compact totally disconnected  group 
$\Delta$ such that the quotient group $G/\Delta$ is the torus $\T^m$. So $G$ has a quotient group 
$\T^{\omega}$.

Case 3. $m=\rank \widehat G$ is finite.  By Theorem~8.22 (10) of \cite{COMPBOOK}, $\rank \widehat {G_0}$ is finite, where $G_0$ denotes the connected component of the identity. By Corollary~8.24\,(i) of \cite{COMPBOOK}, $m=\dim G=\dim G_0$. By Corollary~8.24\,(iii) of \cite{COMPBOOK}, $G_0$ satisfies the second axiom of countability, and so is separable and metrizable. As seen in Theorem~1.5.23 of \cite{AT}, separability is a Three Space Property. So if $G_0$ and $G/G_0$ were  separable, then $G$ would be  separable, which would contradict our assumption that $G$ is not separable. Hence $G/G_0$ must be non-separable. As seen in the proof of Case 1 above, the totally disconnected group $G/G_0$ must be an infinite direct product of cyclic groups of bounded order. Hence $G/G_0$, and also $G$, have a quotient group which, as a countably infinite direct product of cyclic groups of bounded order, is  a countably infinite product of non-trivial compact finite-dimensional Lie groups. 
\end{proof}

\begin{thm}\label{compactconnected} Every non-separable connected  compact group, $G$, 
has a quotient group, $Q$, which is  a countably infinite product of non-trivial compact finite-di\-men\-sional Lie groups. The quotient group, $Q$, is therefore an infinite separable metrizable group. 
\end{thm}

\begin{proof}
By The Sandwich Theorem for Compact Connected Groups, Corollary~9.25, and Theorem~9.24 of \cite{COMPBOOK}, $G$ has a quotient group $G/G'\times \prod\limits_{j\in J} S_j/Z(S_j)$, where $G'$ is the (closed) commutator subgroup of $G$, $J$ is some index set, $S_j$ is a topologically simple simply connected compact Lie group and $Z(S_j)$ is the center of $S_j$, for each $j\in J$.  

Case~1. $J$ is infinite. Then $G$ has $\prod\limits_{j\in J} S_j/Z(S_j)$ as a quotient group. Consequently $G$ has a quotient group which is a countably infinite product of non-trivial compact finite-dimensional Lie groups.

Case~2.  $J$ is finite and  $\rank\widehat{G/G'}$ is finite.  By Theorem~9.52 of \cite{COMPBOOK}, $G'$  is a compact Lie group and so is separable. Suppose that  $G/G'$ is separable. As separability is a Three Space Property by Theorem~1.5.23 of \cite{AT}, this would imply that $G$ is separable which is a contradiction.  So the compact abelian group $G/G'$ is not separable. By Theorem~\ref{compactabelian}, $G/G'$, and hence also $G$, have a quotient group which is  a countably infinite product of non-trivial compact finite-dimensional Lie groups. 

Case~3.  $J$ is finite and  $\rank\widehat{G/G'}$ is infinite. By Proposition~8.15 of \cite{COMPBOOK}, $G/G'$, and hence also $G$, have a quotient group which is  a countably infinite product of non-trivial compact finite-dimensional Lie groups. 
\end{proof}

\begin{rem}\label{discreteabelian} 
{\rm 
No discrete group 
has a quotient group which is a countably infinite product of non-trivial  topological groups since every quotient of a discrete group is evidently discrete. In particular then, a locally compact abelian group need not have a quotient group which is a countably infinite product of non-trivial  topological groups.}
\end{rem}

\begin{thm}\label{connectedLCA} Every non-separable connected locally compact abelian group, $G$, has a quotient group, $Q$, which is a countably infinite product of non-trivial compact finite-dimensional Lie groups. The quotient group, $Q$, is   therefore an infinite separable metrizable group. \end{thm}

\begin{proof} By Theorem~26 of \cite{Mor77} $G$ is isomorphic as a topological group to $\R^n\times K$, where $n$ is a non-negative integer and $K$ is a non-separable compact group. The required result then follows from Theorem~\ref{compactconnected}.
\end{proof}

\begin{thm}\label{compacttotallydisconnected} Every infinite  totally disconnected compact group $G$ has a quotient group, $Q$, which is homeomorphic to a countably infinite product of finite discrete topological groups. The quotient group, $Q$, is thus homeomorphic to the Cantor set and therefore is  an infinite separable metrizable group. 
\end{thm}

\begin{proof} 
As $G$ is a profinite group it is isomorphic as a topological group to a closed subgroup of a product $\prod\limits_{i\in I} F_i$, where $I$ is an index set, and each $F_i$ is a discrete finite group. Let $p_i$ be the projection mapping of $G$ to $F_i$, for $i\in I$. We can assume that $p_i(G)=F_i$, for each $i\in I$. As $G$ is an  infinite group and each $F_i$ is a finite group, for each $y\in G$ and $i\in I$ there is an infinite number of members $z$ of $G$ such that $p_i(y)=p_i(z)$. Fixing $g_1\in G$ and $i_1\in I$, let $t_1=p_{i_1}(g_1)\in F_{i_1}$.  There exists an index $i_2$ and an element $g_2$ in $G$, such that $p_{i_1}(g_1)=p_{i_1}(g_2)$ and  $p_{i_2}(g_1)\ne p_{i_2}(g_2)=t_2\in F_{i_2}$. There are infinitely many $y$ in $G$ such that $p_{i_1}(y) = t_1$ and $p_{i_2}(y)=t_2$. So there exists an index $i_3$ and an element $g_3\in G$ such that $p_{i_1}(g_2)=p_{i_1}(g_3)$, $p_{i_2}(g_2)=p_{i_2}(g_3)$ but $p_{i_3}(g_2)\ne p_{i_3}(g_3)=t_3\in F_{i_3}$. So by induction we obtain infinite sequences of indices $i_1,i_2,i_3,\dots,i_n,\dots$ and elements $g_1,g_2,g_3,\dots,g_n,\dots$ in $G$ such that $p_{i_j}(g_{n+1})=p_{i_j}(g_n)$, $j=1,2,\dots,n$ but $p_{i_{n+1}}(g_{n+1})\ne p_{i_{n+1}}(g_n)$. Therefore the image $p(G)$ of $G$ in the product  $\prod\limits_{j\in\N}F_{i_j}$ is infinite, where $p$ is the projection mapping. 

So $p(G)$ is an infinite separable metrizable totally disconnected compact group. By Theorem~10.40 of \cite{COMPBOOK}, $p(G)$ is a Cantor set, and so is homeomorphic to a countably infinite product of finite discrete topological groups. As $G$ is compact, $p$ is a quotient mapping, which proves the result.
\end{proof}

We now give a positive answer to Problem~\ref{SQCompact} (i), (ii), (iii), and (iv). 

\begin{thm}\label{SQTheoremCompact} \textbf{(Separable Quotient Theorem for Compact Groups)}\quad Let $G$ be an infinite compact group. Then $G$ has a quotient group which is an infinite separable metrizable (compact) group.
\end{thm}

\begin{proof} 
By Corollary~2.43 of \cite{COMPBOOK}, $G$ is a strict projective limit of compact Lie groups. So $G$ is isomorphic as a topological group to a subgroup of a product $\prod\limits_{i\in I} L_i$, where $I$ is an index set and each $L_i$, $i\in I$, is a Lie group. Let $p_i$ be the projection 
map of $G$ into $L_i$, $i\in I$. 

Assume first that each group $p_i(G)$ is finite. Then $G$ is a compact group which is isomorphic as a topological group to a subgroup of a product of finite Lie groups; that is, $G$ is a totally disconnected compact group. The required result then follows from Theorem~\ref{compacttotallydisconnected}. 

So, assume that for some  $n\in I$, $p_n(G)$ is infinite. As $G$ is compact,  $p_n$ is a quotient mapping of $G$ onto the infinite compact metrizable (separable) group $p_n(G)$, as required. 
\end{proof}

\begin{em-remark}\label{compact}
{\rm 
Note that the Kakutani-Kodaira-Montgomery-Zippin Theorem, in the particular case of compact groups, immediately shows that $G$ has a separable metrizable (compact) quotient group. However, this quotient might be finite, and we have to spend additional effort in order to prove there is an infinite one.

Theorem \ref{SQTheoremCompact} for compact groups  is used in the proof of the more general Theorem~\ref{SQSigma-CompactLCTheorem}
dealing with $\sigma$-compact locally compact groups.
}
\end{em-remark}

Having settled the compact group case, we mention three related problems.

\begin{problem}\label{SQsigma-compact}  
\textbf{Separable Quotient Problem for  $\boldsymbol\sigma$-compact  Groups.} Does every   non-discrete $\sigma$-compact  group have a  separable quotient group which is (i) non-trivial; (ii) infinite; (iii) metrizable; (iv) infinite metrizable?
\end{problem}

\begin{problem}\label{SQprecompact}  \textbf{Separable Quotient Problem for Precompact  Groups.} Does every  infinite precompact  group have a  separable quotient group which is (i) non-trivial; (ii) infinite; (iii) non-trivial metrizable; (iv) infinite metrizable? 
\end{problem}

\begin{problem}\label{SQpseudocompact}  \textbf{Separable Quotient Problem for Pseudocompact  Groups.} Does every infinite pseudocompact  group  have a  separable quotient group which is (i) non-trivial; (ii) infinite; (iii) non-trivial metrizable; (iv) infinite metrizable?
\end{problem}

Problems \ref{SQsigma-compact} and \ref{SQpseudocompact} are answered in the next section, while  Problem \ref{SQprecompact} is answered in  Theorem~\ref{Th:Ans}.

\smallskip
We now give a positive answer to Problem~\ref{SQLCA} (i), (ii), (iii), and (iv).  First we state a proposition due to W.\,R.~Scott \cite{Scott}. (See EA1.12 of \cite{COMPBOOK} and 16.13(c) of \cite{HewittRoss}.)

\begin{prop}\label{CountableQuotientGroup} Let $A$ be an uncountable abelian group and $\aleph$ a cardinal number satisfying $\aleph_0\le \aleph < |A|$, where $|A|$ denotes the cardinality of the group $A$. Then $A$ has a subgroup $B$ such that  $|A/B|$, the cardinality of the quotient group $A/B$, equals $\aleph$. In particular, every uncountable abelian group has a quotient group which is countably infinite.
\end{prop}

\begin{thm}\label{SQLCATheorem}\textbf{(Separable Quotient Theorem for Locally Compact Abelian Groups)}\quad Let $G$ be an infinite locally compact abelian group. Then $G$ has a quotient group which is an infinite separable metrizable  group.
\end{thm}

\begin{proof} By the Principal Structure Theorem for Locally Compact Abelian Groups, Theorem~26 
of \cite{Mor77}, $G$ has an open subgroup $H$ isomorphic as a topological group to $\R^n\times K$, where $n$ is a non-negative integer and $K$ is a compact abelian group.

Case 1. $G=H$. Then the theorem follows from Theorem~\ref{SQTheoremCompact}.

Case 2. $G/H$ is an infinite discrete abelian group. Then by Proposition~\ref{CountableQuotientGroup} $G/H$, and hence also $G$,  
has a quotient group which is countably infinite and discrete.

Case 3. $G/H$ is finite. Then $G$ is a compactly-generated locally compact abelian group. By Exercise Set~14 \#3 of \cite{Mor77}, $G$ is isomorphic as a topological group to $\R^n\times \Z^m \times C$, where $n$ and $m$ are non-negative integers and $C$ is a compact abelian group.  The required result then follows from Theorem~\ref{SQTheoremCompact}.
\end{proof}

Recall that a \emph{proto-Lie group} is defined in  \cite[Definition 3.25]{PROBOOK} to be a  topological group $G$ for which every neighborhood of the identity contains a closed normal subgroup $N$ such that the quotient group $G/N$ is a Lie group.  If $G$ is also a complete topological group, then it is said to be a \emph{pro-Lie group}. If $G$ is a proto-Lie group (respectively, pro-Lie group) with all the quotient Lie groups $G/N$ discrete then $G$ is said to be \emph{protodiscrete} (respectively, \emph{prodiscrete}). It is immediately clear that if $G$ is a proto-Lie group which is not a Lie group, then it is not topologically simple.

\begin{thm}\label{SQPro-LieGroupTheorem} \textbf{(Separable Quotient Theorem for Proto-Lie Groups)}\quad Let $G$ be an infinite proto-Lie group which is not protodiscrete; that is, $G$ is not totally disconnected. Then $G$ has a quotient group which is an infinite separable metrizable  (Lie) group.
\end{thm}

\begin{proof}
By definition, $G$ has a quotient group $G/N$ which is a Lie group $L$.  

Case~1. All such $L$ are discrete.  Then $G$ is a subgroup of a product of  discrete groups; that is $G$ is a protodiscrete group. This contradicts the assumption in the theorem.

Case~2. At least one such $L$  is not discrete.  Then $G$ has a quotient group which is a non-discrete Lie group and so is infinite separable and metrizable. 
\end{proof}

\begin{thm}\label{SQSigma-CompactPro-LieGroupTheorem}\hspace{3pt}\textbf{(Separable Quotient Theorem for $\boldsymbol \sigma$-compact Pro-Lie Groups)}\quad Let $G$ be an infinite $\sigma$-compact pro-Lie group. Then $G$ has a quotient group which is an infinite separable metrizable  group.
\end{thm}

\begin{proof} If $G$ is not prodiscrete, the theorem follows immediately from 
Theorem~\ref{SQPro-LieGroupTheorem}. 

So consider the case that $G$ is prodiscrete. Whenever $N$ is a closed  normal subgroup  such that $G/N$ is discrete, the group $G/N$ must also be $\sigma$-compact and therefore countable. If it is countably infinite, then we have immediately that $G$ has an infinite separable metrizable quotient group. If each such $G/N$ is finite, then $G$ is compact (and totally disconnected) and the theorem follows from  Theorem~\ref{SQTheoremCompact}.
\end{proof}

Another significant generalization of \ref{SQLCATheorem} is 
Theorem~\ref{SQAbelianPro-LieGroupTheorem}.

\begin{thm}\label{SQAbelianPro-LieGroupTheorem} \textbf{(Separable Quotient Theorem for Abelian Pro-Lie Groups)}\quad Let $G$ be an infinite abelian pro-Lie group. Then $G$ has a quotient group which is an infinite separable metrizable  group.
\end{thm}

\begin{proof} 
We modify Case~1 in the proof of Theorem~\ref{SQPro-LieGroupTheorem}. If all such $L$ are discrete then we have one of the two cases:\smallskip

Case (i) All such $L_i$ are finite. Then $G$ is isomorphic as a topological group to a closed subgroup of a product of finite discrete groups. So $G$ is compact. Thus by Theorem \ref{SQTheoremCompact}, $G$ has a quotient group which is an infinite separable metrizable group.

Case (ii) At least one such $L$ is infinite and discrete. So $G$ has a quotient group {which} is an infinite discrete group. Thus by Theorem \ref{SQLCATheorem}, $G$ has a quotient group {which} is an infinite separable metrizable group, {which} completes the proof of the theorem. 
\end{proof}

The next theorem, {which} generalizes Theorem~\ref{SQTheoremCompact}, provides a partial but significant answer to Problem~\ref{SQLC}.

\begin{thm}\label{SQSigma-CompactLCTheorem} \textbf{(Separable Quotient Theorem for 
$\boldsymbol{\sigma}$-compact Locally Compact Groups)} Every infinite $\sigma$-compact locally compact group has a quotient group which is an infinite separable metrizable group. 
\end{thm}

\begin{proof} 
By the Kakutani-Kodaira-Montgomery-Zippin Theorem, every $\sigma$-compact locally compact group, $G$, has a compact normal subgroup, $K$, such that $G/K$ is a separable metrizable group. (See Theorem~2 of \cite{Hu} for a more general statement and elegant proof.)

Case~1.  $G/K$ is finite. Then $G$ must be compact as compactness is a Three Space Property by Theorem~5.25 of \cite{HewittRoss}. The required result then follows from 
Theorem~\ref{SQTheoremCompact}.

Case~2. $G/K$ is infinite. Then $G$ has a quotient group which is infinite separable and metrizable, as required. 
\end{proof}

Pierre-Emmanuel Caprace  studied  the class $\mathscr S$ 
of all non-discrete compactly-generated locally compact groups 
that are topologically simple.  He writes:
 {\it \lq\lq{Simple Lie groups and simple algebraic groups over local fields are the most prominent members of the class $\mathscr S$\rq\rq}}
\cite{Capracepreprint}.

As an immediate corollary of Theorem~\ref{SQSigma-CompactLCTheorem}, we observe the following apparently new result which reveals the topological structure of members of the class $\mathscr S$.

\begin{prop}\label{Class_S} Every topological group in the class $\mathscr S$ is a separable metrizable group.
\end{prop}

\begin{proof} 
A quotient group of any $G \in \mathscr S$ is trivial or it is $G$ itself. But every 
compactly-generated topological group is $\sigma$-compact, and therefore by Theorem~\ref{SQSigma-CompactLCTheorem}, every $G \in \mathscr S$ has a quotient 
group which is an infinite separable metrizable group.  So $G$ itself is an infinite separable 
metrizable group.
\end{proof}

Recall that a topological group $G$ is said to be \emph{almost connected} if the quotient group $G/G_0$ is compact, where $G_0$ denotes the connected component of the identity of $G$. As every almost connected locally compact group is $\sigma$-compact, we immediately obtain the following corollary.

\begin{corollary}\label{SQAlmostConnectedTheorem} \textbf{(Separable Quotient Theorem for Almost Connected Locally Compact Groups)}\quad  Every infinite almost connected  locally compact group has a quotient group which is an infinite separable metrizable group. 
\end{corollary}

The class of \emph{$\R$-factorizable} groups (see \cite[Chapter~8]{AT}) contains all pseudocompact groups as well as $\sigma$-compact groups. We note that by Theorem~8.1.9 of \cite{AT}, a locally compact group is $\R$-factorizable if and only if it is $\sigma$-compact. This suggests the following problem which will be answered in Section~\ref{Sec:2}.

\begin{problem}\label{SQRfactorizable} \textbf{Separable Quotient Problem for $\boldsymbol \R$-factorizable  Groups.} Does every $\R$-factorizable group have a separable quotient group which is (i) non-trivial; (ii) infinite; (iii) metrizable; (iv) infinite metrizable?
\end{problem}

\section{$\sigma$-compact groups, Lindel\"of $\Sigma$-groups and pseudocompact groups}

\noindent Here we give the positive answer to (i) and (ii) of Problem~\ref{SQsigma-compact} 
for the wider class of \emph{Lindel\"of $\Sigma$-groups} and then answer (iii) and (iv) of Problem~\ref{SQsigma-compact} in the negative.
Recall that the class of Lindel\"of $\Sigma$-groups 
contains all $\sigma$-compact and all separable metrizable topological groups, and is closed with respect to countable products, closed subgroups and continuous homomorphic images (see \cite[Section~5.3]{AT}).

\smallskip
 Proposition~\ref{Prop:M1} answers (i) and (ii) of Problem~\ref{SQsigma-compact} 
in a stronger form. First we present an important fact about Lindel\"of $\Sigma$-groups
(see \cite[Lemma~5.3.24]{AT}).

\begin{lemma}\label{Le:Usp}
Let $G$ be a Lindel\"of $\Sigma$-group. Then
\begin{enumerate}
\item[{\rm (a)}] for every closed normal subgroup $N$ of type $G_\delta$ in $G$, 
the quotient group $G/N$ has a countable network;
\item[{\rm (b)}] the sets of the form $\pi^{-1}(V)$ constitute a base of $G$, where $\pi\colon G\to K$ 
is an open continuous homomorphism onto a topological group $K$ with a countable network 
and $V$ open in $K$.
\end{enumerate}
\end{lemma}

\begin{prop}\label{Prop:M1} {\bf (Separable Quotient Theorem for Lindel\"of $\Sigma$-groups)}
Let $G$ be an infinite Lindel\"of $\Sigma$-group. Then $G$ has a quotient group which is infinite and separable. Indeed, the topology of $G$ is initial with
respect to the family of quotient homomorphisms onto infinite groups with a countable 
network. 
\end{prop}

\begin{proof}
It follows from (b) of Lemma~\ref{Le:Usp} that the topology of $G$ is initial with respect 
to the family of quotient homomorphisms onto topological groups with a countable 
network. Also, since the singletons in a space of countable pseudocharacter are 
$G_\delta$-sets, (b) of the same lemma implies that every neighborhood of the identity 
in $G$ contains a closed normal subgroup of type $G_\delta$ in $G$. Therefore, as
$G$ is infinite, there exists a closed normal subgroup $N_0$ of $G$ of type $G_\delta$
such that $G/N_0$ is infinite. 

Finally, let $\mathcal{N}$ be the family of closed normal subgroups $N$ of type $G_\delta$
in $G$ with $N\subset N_0$. Then the family of quotient homomorphisms $\pi_N\colon G
\to G/N$, with $N\in\mathcal{N}$, generates the topology of $G$ and each group $G/N$
is infinite. According to (a) of Lemma~\ref{Le:Usp}, the group $G/N$ has a countable
network for each $N\in\mathcal{N}$.
\end{proof}

Since every $\sigma$-compact topological group is evidently a Lindel\"of $\Sigma$-group,
the next result is immediate from Proposition~\ref{Prop:M1}.

\begin{corollary}\label{Cor:sigma-c} {\bf (Separable Quotient Theorem for  $\sigma$-compact Groups)} Let $G$ be an infinite $\sigma$-compact topological group. Then $G$ has a quotient group which is infinite and separable. Indeed,
the topology of $G$ is initial with respect to
the family of quotient homomorphisms of the group onto infinite groups with a countable network.
\end{corollary}

In Proposition  \ref{Ex:Qs} we give a negative answer to (iii) and (iv) of Problem~\ref{SQsigma-compact} by producing a $\sigma$-compact group which has no non-trivial separable metrizable quotient group. 

\begin{prop}\label{Ex:Qs}
There exists a countably
 infinite precompact abelian group $H$ such that every quotient 
group of $H$ is either trivial or non-metrizable.
\end{prop}

\begin{proof}
For a given prime number $p$, denote by $\mathbb{C}_p$ the quasicyclic $p$-group 
$$
\{z\in\T: z^{p^n}=1 \mbox{ for some } n\in\N\}
$$
considered as a subgroup of the group $\T$. Clearly $\mathbb{C}_p$ is a countable 
infinite abelian group. Let $\tau$ be the \emph{Bohr topology} of $\mathbb{C}_p$, 
i.e.~the maximal precompact topological group topology of $\mathbb{C}_p$ \cite[Section~9.9]{AT}. 
We claim that the group $H=(\mathbb{C}_p,\tau)$ is as required. Indeed, let $N$ be a proper
subgroup of $H$. Then there exists an integer $n\geq 0$ such that 
$$
N=\{z\in \mathbb{C}_p: z^{p^n}=1\}
$$
(see \cite[Chapter~4]{Robinson}), so $N$ is finite and, hence, closed in $H$. Therefore
the quotient group $H/N$ is infinite and, in fact, algebraically isomorphic to $\mathbb{C}_p$.
It follows from \cite[Proposition~9.9.9\,c)]{AT} that the quotient topology of $H/N$ and the 
Bohr topology of the abstract group $H/N$ coincide. We conclude, therefore, that the groups 
$H$ and $H/N$ are  isomorphic as topological groups. The group $H/N$ being infinite and
precompact is not discrete. Further, according to \cite[Theorem~9.9.30]{AT}, all compact 
subsets of the groups $H$ and $H/N$ are finite, so the space $H/N$ does not contain 
non-trivial convergent sequences. This implies that $H/N$ is not metrizable.
\end{proof}

We now answer Problem~\ref{SQpseudocompact} in the affirmative. 

\begin{thm}\label{Th:Pseu} {\bf (Separable Quotient Theorem for Pseudocompact Groups)}
The topology of every infinite pseudocompact topological group, $G$, is initial with respect to 
the family of quotient homomorphisms onto infinite compact metrizable groups.  In particular,  $G$ has a quotient group which is  infinite separable compact and metrizable.
\end{thm}

\begin{proof}
The completion, $\varrho{G}$, of the  infinite pseudocompact topological group $G$ is a compact  group that contains $G$ as a dense
topological subgroup. Hence $\varrho{G}$ is a Lindel\"of $\Sigma$-group and we can 
apply Proposition~\ref{Prop:M1} to conclude that the topology of $\varrho{G}$ is
initial with respect to the family of quotient homomorphisms onto infinite topological
groups with a countable network. 

Let $p\colon \varrho{G}\to K$ be an open continuous homomorphism of $\varrho{G}$
onto an infinite topological group $K$ with a countable network. Then the group $K$
is compact and, hence, has a countable base. In particular $K$ is metrizable. Denote
by $\pi$ the restriction of $p$ to $G$. Since the group $G$ is pseudocompact it
meets every nonempty $G_\delta$-set in $\varrho{G}$ \cite{CoRo}. The points in 
$K$ are $G_\delta$-sets, and so are the fibers $p^{-1}(y)$ in $\varrho{G}$, for all
$y\in K$. Therefore $G\cap p^{-1}(y)\neq\emptyset$ for each $y\in K$, which implies
the equality $p(G)=K$. It is also clear that $G\cap\ker{p}$ is dense in $\ker{p}$, for
$\ker{p}$ is a $G_\delta$-set in $\varrho{G}$. Hence the restriction of $p$ to $G$
is an open continuous homomorphism of $G$ onto $K$, by \cite[Theorem~1.5.16]{AT}.
Since the group $K$ is infinite, compact and metrizable, this completes the proof.
\end{proof}

\begin{em-remark}\label{pseudocompactandcompact}
{\rm Theorem~\ref{Th:Pseu} provides an alternative proof of Theorem~\ref{SQTheoremCompact}
for compact groups.}
\end{em-remark}

In the next section we will show that Theorem~\ref{Th:Pseu} cannot be extended to
precompact topological groups, even in the weak form of the existence of
nontrivial separable quotients.


\section{A Precompact Counterexample}\label{Sec:2}

\noindent In this section, firstly we
address Problem~\ref{SQprecompact} and Problem~\ref{SQRfactorizable} 
and answer both in the negative by producing a counterexample. This also provides 
a negative answer to each of Problem~\ref{SQTopGps}, Problem~\ref{SQInfiniteTopGps}, Problem~\ref{SQMetrizableTopGps}, and Problem~\ref{SQInfiniteMetrizableTopGps}. We 
conclude the section with Theorem~\ref{Th:CUD} which yields a negative answer to each of Problem~\ref{SQGsigma} (i), (ii), (iii) and (iv).


We shall now present an example of a \lq\lq{weird\rq\rq} precompact topological group without
non-trivial separable quotients. Furthermore, every subgroup of $G$ is either countable 
and closed or uncountable and dense in $G$. 

Recall that a subset $X$ of an abelian group $G$ with  identity $0_G$ is 
called \emph{independent} if the equality $n_1x_1+\cdots+n_kx_k=0_G$, where 
$n_1,\ldots,n_k\in\Z$ and $x_1,\ldots,x_k$ are pairwise distinct elements of $X$, 
implies that $n_1x_1=\cdots=n_kx_x=0_G$.

The conclusion of the following lemma is well-known, it follows from the fact that 
the dual $\Pi^\wedge$ of the product $\Pi=\prod_{i\in I} K_i$ of compact abelian
groups $K_i$ is isomorphic to the discrete group $\bigoplus_{i\in I} K_i^\wedge$, 
the direct sum of the discrete dual groups $K_i^\wedge$ (see \cite[Theorem~17]{Mor77} 
or \cite[Proposition~9.6.25]{AT}; the latter result is dual to the lemma).  

\begin{lemma}\label{Le:char}
Let $A$ be a nonempty set and $\chi$ a continuous character on the group $\T^A$. 
Then one can find pairwise distinct indices $\alpha_1,\ldots,\alpha_k\in A$ and integers $n_1,\ldots,n_k$ such that $\chi(x)=\prod_{i=1}^k x(\alpha_i)^{n_i}$, for each $x\in\T^A$.
\end{lemma}

\begin{lemma}\label{Le:M}
Let $G$ be an uncountable precompact abelian group such that every countable
subgroup of $G$ is closed. Then the following are equivalent:
\begin{enumerate} 
\item[\rm{(a)}] every uncountable subgroup of $G$ is dense in $G$;
\item[\rm{(b)}] the kernel of every non-trivial continuous character of $G$ is countable.
\end{enumerate}
Furthermore, each of the items (a), (b) implies 
\begin{enumerate}
\item[\rm{(c)}] every quotient group of $G$ is either trivial or non-separable. 
\end{enumerate}
\end{lemma}

\begin{proof}
(a)\,$\Rightarrow$\,(b). Assume that every uncountable subgroup of $G$ is dense in $G$. 
Let $\chi$ be a continuous non-trivial character on $G$. Then the kernel $K$ of $\chi$ 
must be countable\,---\,otherwise $K=G$ and $\chi$ is trivial.\smallskip

(b)\,$\Rightarrow$\,(a). Let $D$ be an uncountable subgroup of $G$. Then 
$K=\mbox{cl}_G(D)$ is a closed subgroup of $G$. If $K\neq G$, then the quotient 
group $G/K$ is non-trivial and precompact, so there exists a non-trivial continuous 
character $\psi$ on $G/K$. Let $p\colon G\to G/K$ be the quotient homomorphism. 
Then $\chi=\psi\circ{p}$ is a non-trivial continuous character on $G$ and $K\subset 
\ker\chi$. This contradicts (b) of the lemma and proves that $D$ is dense in $G$.\smallskip

We have thus proved that (a) and (b) of the lemma are equivalent.\smallskip

(b)\,$\Rightarrow$\,(c). Let $h\colon G\to H$ be a continuous open homomorphism of $G$ 
onto a topological group $H$ with $|H|>1$. Then $H$ is precompact and abelian. Since 
$|H|>1$, there exists a non-trivial continuous character $\chi_H$ on $H$. Hence 
$\chi_G=\chi_H\circ{h}$ is a non-trivial continuous character on $G$. By (b) of the lemma, 
the kernel of $\chi_G$ is countable. Therefore the kernel of $h$ is countable as well. 
This implies that $|H|=|G|>\omega$.

Let $C$ be a countable subgroup of $H$. Then $h^{-1}(C)$ is a countable subgroup 
of $G$, so $h^{-1}(C)$ is closed in $G$ by (a) of the lemma. Since the homomorphism 
$h$ is continuous and open, $C$ is closed in $H$. We have thus proved that all countable
subgroups of $H$ are closed. If $S$ is a countable subset of $H$, then the subgroup 
$C$ of $H$ generated by $S$ is also countable and, hence, $C$ is closed in $H$. 
Since the group $H$ is uncountable, we conclude that it cannot be separable. 
\end{proof}

The proof of the following fact is quite simple and left to the reader.

\begin{lemma}\label{Le:Alg}
Let $A,B,C$ be subgroups of an abelian group $K$. 
\begin{enumerate}
\item[{\rm (a)}] The equality $(A+B)\cap C=(A^*+B)\cap C$ holds, where $A^*=A\cap (B+C)$. 
\item[{\rm (b)}] If $A$ and $B+C$ have trivial intersection, then $(A+B)\cap C=B\cap C$.
\end{enumerate}
\end{lemma}

Notice that item (b) of the above lemma is immediate from (a).\smallskip

Following \cite[Section~4]{Tk88} we say that a subgroup $S$ of an abelian topological 
group $G$ is \emph{$h$-embedded} in $G$ if \emph{every} homomorphism $f$ of 
$S$ to the circle group $\T$ extends to a \emph{continuous} character of $G$. In particular, 
every homomorphism of an $h$-embedded subgroup of $G$ to $\T$ is continuous.
If the group $G$ is precompact, then a subgroup $S$ of $G$ is $h$-embedded
if and only if $S$ inherits from $G$ the maximal precompact group topology, i.e.~the 
\emph{Bohr} topology.\smallskip 

The following fact is a weaker version of \cite[Proposition~2.1]{ACDT}.

\begin{lemma}\label{Le:ACDT}
If every countable subgroup of an abelian topological group $G$ is $h$-embedded 
in $G$, then the countable subgroups of $G$ are closed.
\end{lemma}

In the next theorem we present an uncountable precompact group without non-trivial
separable quotients.

\begin{thm}\label{Th:Ans}
There exists an uncountable dense subgroup $G$ of the compact group $\T^\cont$ 
such that every countable subgroup of $G$ is $h$-embedded in $G$ and closed and 
every uncountable subgroup of $G$ is dense in $G$. Hence every quotient group of 
$G$ is either trivial or non-separable. 
\end{thm}

\begin{proof}
Let $X$ be a set of cardinality $\cont$ and $A(X)$ the free abelian group on $X$.
We will define a monomorphism $f$ of the group $A(X)$ to $\T^\cont$ and then take 
$G$ to be $f(A(X))$, considered as a topological subgroup of $\T^\cont$.

Since $|A(X)|=|X|=\cont$, the group $A(X)$ contains exactly $\cont^\omega=\cont$ 
countable subgroups. Let $\mathcal{C}$ be the family of all countable subgroups 
of $A(X)$. For every $C\in\mathcal{C}$, the family $H(C)$ of all homomorphisms of 
$C$ to $\T$ has cardinality at most $\cont$, so the family $\mathcal{H}=\bigcup\{H(C): 
C\in\mathcal{C}\}$ satisfies $|\mathcal{H}|=\cont$. 

Let $\{(C_\alpha,g_\alpha): \alpha<\cont\}$ be an enumeration of the family 
$\{(C,g): C\in\mathcal{C},\ g\in H(C)\}$. The subgroup $C_\alpha$ of $A(X)$ 
being countable, there exists a countable subset $Y_\alpha$ of $X$ 
such that $C_\alpha\subset \hull{Y_\alpha}$, where $\alpha<\cont$. 
Let $Z_\alpha=X\setminus Y_\alpha$.

Further, let $E$ be an independent subset of $\T$ such that $|E|=\cont$ and each element 
of $E$ has infinite order (see \cite[Lemma~7.1.6]{AT}). Let also $\{E_\alpha: \alpha<\cont\}$
be a partition of $E$ into $\cont$ pairwise disjoint subsets $E_\alpha$, each of cardinality 
$\cont$. 

For every $\alpha<\cont$, we will define a homomorphism $f_\alpha\colon A(X)\to \T$
satisfying the following conditions:
\begin{enumerate}
\item[{\rm (i)}] $f_\alpha$ and $g_\alpha$ coincide on the subgroup $C_\alpha$ of $A(X)$;
\item[{\rm (ii)}] $b_\alpha=f_\alpha\res Z_\alpha$ is a bijection of $Z_\alpha$ onto a subset 
                       of $E_\alpha$, so the restriction of $f_\alpha$ to $\hull{Z_\alpha}$ is a 
                       monomorphism;
\item[{\rm (iii)}] the subgroups of $\T$ generated by the sets $f_\alpha(Z_\alpha)$ and 
                        $f_\alpha(Y_\alpha)\cup\bigcup_{\nu<\alpha} f_\nu(X)$ have
                        trivial intersection;
\item[{\rm (iv)}] $f_\alpha(A(X))\subset \hull{f_\alpha(Y_\alpha)}\cdot\hull{E_\alpha}$.
\end{enumerate}

Notice that (ii) and (iii) together imply that $\ker f_\alpha\subset \hull{Y_\alpha}$, 
so the kernel of $f_\alpha$ is countable for each $\alpha<\cont$. Also, since the
equality $X=Y_\alpha\cup Z_\alpha$ holds for each $\alpha<\cont$, (iv) follows 
from (iii). However we isolate (iv) for further applications. To construct the family 
$\{f_\alpha: \alpha<\cont\}$ satisfying (i)--(iv) we argue as follows. 

For every $\alpha<\cont$, we extend $g_\alpha$ to a homomorphism $h_\alpha$
of $\hull{Y_\alpha}$ to $\T$. There exists a countable set $D_0\subset E$ such that
$\hull{h_0(Y_0)}\cap \hull{E}\subset \hull{D_0}$. Let $F_0=E_0\setminus D_0$.
Then $|F_0|=\cont$, so we choose a bijection $b_0$ of $Z_0=X\setminus Y_0$ onto 
$F_0$. Since $A(X)=\hull{Y_0}\oplus \hull{Z_0}$, there exists a homomorphism
$f_0\colon A(X)\to\T$ which extends $h_0$ and coincides with $b_0$ on $Z_0$. 
Clearly the restriction of $f_0$ to $\hull{Z_0}$ is a monomorphism. One can easily 
verify that the groups $\hull{f_0(Y_0)}$ and $\hull{F_0}=\hull{f_0(Z_0)}$ have trivial 
intersection. Therefore $f_0$ satisfies (i)--(iv). 

Assume that for some $\alpha<\cont$, we have defined a family $\{f_\nu: \nu<\alpha\}$ 
of homomorphisms of $A(X)$ to $\T$ satisfying (i)--(iv). The subgroup $H_\alpha$ of 
$\T$ generated by $h_\alpha(Y_\alpha)\cup\bigcup_{\nu<\alpha} f_\nu(Y_\nu)$ 
has cardinality at most $|\alpha+1|\cdot\omega<\cont$. Let $H^*_\alpha=H_\alpha\cap 
\big\langle{\bigcup_{\nu\leq\alpha} E_\nu}\big\rangle$. Then $|H^*_\alpha|\leq |H_\alpha|<\cont$, 
so we can find a subset $D_\alpha$ of $\bigcup_{\nu\leq\alpha} E_\nu$ with $|D_\alpha|<\cont$ 
such that $H^*_\alpha\subset\hull{D_\alpha}$. Let $F_\alpha=E_\alpha\setminus D_\alpha$. 
Notice that $|F_\alpha|=\cont$.\medskip 

\noindent
\textbf{Claim~1.} \textit{The groups $\hull{F_\alpha}$ and 
$H_\alpha\cdot \big\langle{\bigcup_{\nu<\alpha} E_\nu}\big\rangle$ 
have trivial intersection.}\medskip 

Indeed, we apply (a) of Lemma~\ref{Le:Alg} with $A=H_\alpha$, 
$B=\big\langle{\bigcup_{\nu<\alpha} E_\nu}\big\rangle$, and $C=\hull{F_\alpha}$ to deduce that 
\begin{eqnarray}
(A\cdot B)\cap C=(A^*\cdot B)\cap C \subset
\Big(H^*_\alpha\cdot \big\langle{\bigcup_{\nu<\alpha} E_\nu}\big\rangle\Big) 
\cap \hull{F_\alpha},\label{Eq:3}
\end{eqnarray}
where $A^*=A\cap (B\cdot C)\subset H^*_\alpha$. Further, $F_\alpha$ and 
$D_\alpha\cup\bigcup_{\nu<\alpha} E_\nu$ are pairwise disjoint subsets of 
the independent set $E\subset \T$, so the groups $\hull{F_\alpha}$ and 
$\hull{D_\alpha}\cdot \big\langle{\bigcup_{\nu<\alpha} E_\nu}\big\rangle$ 
have trivial intersection. Since $H^*_\alpha\subset \hull{D_\alpha}$, it follows 
from $(\ref{Eq:3})$ that $(A\cdot B)\cap C=\{1\}$. This proves Claim~1.

Consider an arbitrary bijection $b_\alpha\colon Z_\alpha\to F_\alpha$. There exists a 
homomorphism $f_\alpha\colon A(X)\to\T$ which extends $h_\alpha$ and coincides 
with $b_\alpha$ on $Z_\alpha$. It is clear that $f_\alpha$ extends $g_\alpha$ and 
the restriction of $f_\alpha$ to $\hull{Z_\alpha}$ is a monomorphism. Hence conditions 
(i) and (ii) hold true at the step $\alpha$. Condition (iv) is also fulfilled. Indeed, it follows
from the definition of $f_\alpha$ that $f_\alpha(Z_\alpha)=F_\alpha\subset E_\alpha$.
Hence, applying the equality $X=Y_\alpha\cup Z_\alpha$ we deduce that 
$f_\alpha(A(X))=\hull{f_\alpha(X)}\subset \hull{f_\alpha(Y_\alpha)}\cdot \hull{E_\alpha}$. 
It remains to verify that (iii) is also valid. Indeed, according to (iv) we have 
$f_\nu(X)\subset\hull{f_\nu(Y_\nu)}\cdot \hull{E_\nu}=\hull{f_\nu(Y_\nu)\cup E_\nu}$, 
for each $\nu<\alpha$. Hence the group generated by $f_\alpha(Y_\alpha)\cup
\bigcup_{\nu<\alpha} f_\nu(X)$ (see condition (iii)) is contained in 
\begin{eqnarray*}
\Big\langle{f_\alpha(Y_\alpha)\cup\bigcup_{\nu<\alpha} \big(f_\nu(Y_\nu)\cup E_\nu\big)}
\hskip-0.5pt\Big\rangle \hskip-6.5pt
& = &\hskip-6.5pt
\Big\langle{\bigcup_{\nu\leq\alpha} f_\nu(Y_\nu)\cup\bigcup_{\nu<\alpha} E_\nu}\Big\rangle\\
& = & \hskip-6.5pt \Big\langle{\bigcup_{\nu\leq\alpha} f_\nu(Y_\nu)}\Big\rangle\cdot 
\Big\langle{\bigcup_{\nu<\alpha} E_\nu}\Big\rangle =
H_\alpha\cdot \Big\langle{\bigcup_{\nu<\alpha} E_\alpha}\Big\rangle.
\end{eqnarray*}
By Claim~1, the groups $\hull{F_\alpha}=\hull{f_\alpha(Z_\alpha)}$ and 
$H_\alpha\cdot \big\langle{\bigcup_{\nu<\alpha} E_\alpha}\big\rangle$ have trivial 
intersection. Clearly this implies (iii) and finishes our recursive construction.\smallskip

Let $f$ be the diagonal product of the family $\{f_\alpha: \alpha<\cont\}$. We claim that
$f$ is a monomorphism of $A(X)$ to $\T^\cont$. Indeed, take an arbitrary element $x\in
A(X)$ distinct from the identity and let $C_x$ be the cyclic subgroup of $A(X)$ generated
by $x$. Take a homomorphism $g$ of $C_x$ to $\T$ such that $g(x)\neq 1$. There exists
$\alpha<\cont$ such that $(C_\alpha,g_\alpha)=(C_x,g)$. Then (i) implies that 
$f_\alpha(x)=g_\alpha(x)=g(x)\neq 1$, so the element $f(x)$ is distinct from the identity
of the group $\T^\cont$. This proves our claim. 

We consider the group $G=f(A(X))$ with the topology inherited from the compact group 
$\T^\cont$. Let $p_\alpha$ be the projection of $\T^\cont$ to the $\alpha$th factor 
$\T_{(\alpha)}$, where $\alpha<\cont$. Our definition of $f$ implies that the equality 
$p_\alpha\circ f=f_\alpha$ holds for each $\alpha<\cont$. Since $f$ is a monomorphism 
and $\ker f_\alpha$ is countable, the restriction of each projection $p_\alpha$ to $G$ has 
countable kernel as well.

To show that every countable subgroup, say, $D$ of $G$ is $h$-embedded, we  
consider a homomorphism $h\colon D\to\T$. Take a countable subgroup $C$ of 
$G$ such that $f(C)=D$ and let $g=h\circ f$. There exists $\alpha<\cont$ such that 
$(C,g)=(C_\alpha,g_\alpha)$. Then, for every $x\in C$, we have the equalities
$$
(h\circ f)(x)=g(x)=g_\alpha(x)=f_\alpha(x)=(p_\alpha\circ f)(x).
$$ 
Since $f$ is a monomorphism, we conclude that $h$ and $p_\alpha$ coincide on
$D=f(C)$. Hence $p_\alpha\hskip-1.3pt\res G$ is a continuous character on $G$ 
extending $h$. It follows from Lemma~\ref{Le:ACDT} that every countable subgroup 
of $G$ is closed. 

Our next step is to show that $G$ is dense in $\T^\cont$. This is equivalent to
the density of the projection of $G$ to every subproduct $\T^A$, where $A$ is 
a finite subset of the index set $\cont$. So let $A\subset \cont$ be a finite nonempty
set. Then $Y=\bigcup_{\alpha\in A} Y_\alpha$ is a countable subset of $X$, so
we can take an element $x\in X\setminus Y$. It follows from condition (ii) of our 
construction that $f_\alpha(x)\in E_\alpha$, for each $\alpha\in A$. Since the sets 
$E_\alpha\subset E$ with $\alpha\in A$ are pairwise disjoint, we see that the 
coordinates $\{f_\alpha(x): \alpha\in A\}$ of the element $p_A(f(x))\in \T^A$ are 
pairwise distinct and form a subset of $E$; here $p_A\colon \T^\cont\to \T^A$ is 
the projection. Notice that the set $E$ is independent and consists of elements 
of infinite order. Hence the element $p_A(f(x))$ generates a dense subgroup of 
$\T^A$ (this follows, e.g., from \cite[Example~65]{Pon}). Therefore $p_A(G)$ is 
dense in $\T^A$, as claimed.

To show that every non-trivial quotient group of $G$ is not separable it suffices,
by Lemma~\ref{Le:M}, to verify that the kernel of every non-trivial continuous 
character of $G$ is countable. 

Let $\varphi\colon G\to\T$ be a continuous homomorphism with $|\varphi(G)|>1$.
Then $\varphi$ extends to a continuous homomorphism $\psi\colon \T^\cont\to\T$
(see \cite[Lemma~2.2]{GTBH}). By Lemma~\ref{Le:char}, there exist pairwise distinct 
indices $\alpha_1,\ldots,\alpha_k\in\cont$ and non-zero integers $n_1,\ldots,n_k$ such 
that $\psi(x)= \prod_{i=1}^k x(\alpha_i)^{n_i}$, for each $x\in \T^\cont$. We can assume
that $\alpha_1<\cdots<\alpha_k$. Since $p_\alpha \circ f = f_\alpha$ for each 
$\alpha\in\cont$, the expression for $\psi(x)$ can be rewritten as follows:
\begin{eqnarray}
\varphi(f(b))=\psi(f(b))=\prod_{i=1}^k f_{\alpha_i}(b)^{n_i}, \mbox{ for each } 
b\in A(X).\label{Eq:1}
\end{eqnarray}

Clearly the set $Y=\bigcup_{i=1}^k Y_{\alpha_i}\subset X$ is countable.\medskip 

\noindent
\textbf{Claim~2.} \textit{The kernel of the homomorphism $\varphi\circ f$ is contained in 
$\hull{Y}$.}\medskip

Indeed, take an arbitrary element $b\in A(X)\setminus \hull{Y}$. Let $Z=X\setminus Y$. 
Then $b=y+z$, where $y\in \hull{Y}$ and $z\in \hull{Z}$. Clearly $z$ is distinct from the 
identity of $A(X)$. For simplicity, let $\alpha=\alpha_k$. It follows from $(\ref{Eq:1})$ that
\begin{eqnarray}
\varphi(f(b)) = f_\alpha(z)^{n_k}\cdot f_\alpha(y)^{n_k}\cdot \prod_{i=1}^{k-1} f_{\alpha_i}(y)^{n_i} 
\cdot \prod_{i=1}^{k-1} f_{\alpha_i}(z)^{n_i}.\label{Eq:2}
\end{eqnarray}
Since $Z\subset Z_\alpha$ and $f_\alpha\hskip-2pt\res Z_\alpha$ is a
monomorphism, we see that $f_\alpha(z)^{n_k}\neq 1$.

Let us put 
$$
A=\hull{f_\alpha(Y_\alpha)\cup\bigcup_{i=1}^{k-1} f_{\alpha_i}(X)},\ 
B=\hull{f_\alpha(Y\setminus Y_\alpha)}, \mbox{ and } C=\hull{f_\alpha(Z)}.
$$
Since $Y=Y_\alpha\cup (Y\setminus Y_\alpha)$ and $y\in \hull{Y}$, the element 
on the right hand side of equality $(\ref{Eq:2})$ is in $C\cdot B\cdot A$. Suppose 
for a contradiction that $\varphi(f(b))=1$. Then, by $(\ref{Eq:2})$, we have that 
$1\neq f_\alpha(z)^{n_k}\in C\cap (B\cdot A)$. Further, the group $B\cdot C$ is
algebraically generated by the set 
$$
f_\alpha(Y\setminus Y_\alpha)\cup f_\alpha(X\setminus Y)=
f_\alpha(X\setminus Y_\alpha)=f_\alpha(Z_\alpha),
$$ 
so condition (iii) of our construction implies that $A\cap (B\cdot C)=\{1\}$. Hence, by
(b) of Lemma~\ref{Le:Alg}, the equality $(A\cdot B)\cap C=B\cap C$ holds. Since 
$Y\setminus Y_\alpha$ and $X\setminus Y$ are disjoint subsets of the independent 
set $Z_\alpha$, condition (ii) implies that the groups $B$ and $C$ have trivial intersection. 
Hence the intersection $(B\cdot A)\cap C$ is trivial. This contradiction shows that 
kernel of $\varphi\circ{f}$ is a subgroup of $\hull{Y}$, which proves Claim~2. 

Since the set $Y$ is countable, Claim~2 implies that $|\ker (\varphi\circ{f})|\leq\omega$.
We already know that $f$ is a monomorphism, so the kernel of $\varphi$ is also
countable. Hence every non-trivial quotient of the group $G=f(A(X))$ is not separable,
according to the implication (b)\,$\Rightarrow$\,(c) in Lemma~\ref{Le:M}.
\end{proof}

We now give an answer to Problem~\ref{SQRfactorizable}.

\begin{corollary}\label{coro:Rfac}
There exist infinite $\mathbb{R}$-factorizable groups without non-trivial separable or
metrizable quotients.
\end{corollary}

\begin{proof}
Since the group $G$ in Theorem~\ref{Th:Ans} is 
precompact, it is $\mathbb{R}$-factorizable according to 
\cite[Corollary~8.1.17]{AT}.
Every quotient of $G$ is also precompact.
Notice that 
every precompact metrizable group is separable (this follows e.g. from 
\cite[Proposition~3.4.5]{AT}).
Therefore the precompact group $G$ does not have 
non-trivial separable or metrizable quotients. 
\end{proof}

Our next aim is to show that every power of the group $G$ in Theorem~\ref{Th:Ans}
does not have non-trivial separable quotients. The proof of this fact requires several
preliminary results.

\begin{lemma}\label{Le:ML}
Let $H$ be an abelian topological  group, $k\in\N$, and $n_1,\ldots,n_k$ be integers,
not all equal to zero, such that the highest common divisor of $n_1,\ldots,n_k$ 
equals $1$. Let also $\varphi$ be a homomorphism of $H^k$ to $H$ defined by 
$\varphi(x_1,\ldots,x_k)=n_1x_1+\cdots+n_kx_k$, for each $(x_1,\ldots,x_k)\in H^k$. 
Then the homomorphism $\varphi$ is open and surjective. In particular, the quotient 
group $H^k/N$ is isomorphic as a topological group to $H$, where $N=\ker\varphi$.
\end{lemma}

\begin{proof}
It is clear that the homomorphism $\varphi$ is continuous. Since $\varphi(H^k)=H$, 
it suffices to verify that $\varphi$ is open. It follows from our assumption about 
$n_1,\ldots,n_k$ that there exist integers $m_1,\ldots,m_k$ such that $m_1n_1+
\cdots+m_kn_k=1$. Consider the homomorphism $\lambda$ of $H$ to $H^k$ 
defined by $\lambda(x)=(m_1x,\ldots,m_kx)$. The continuity of $\lambda$ 
is evident. Notice that $\varphi\circ\lambda=Id_H$. In particular, $\varphi$
is surjective, as we mentioned it above. 

Take an arbitrary open neighborhood $U$ of the identity element in $H^k$. We 
can assume that $U=V\times\cdots\times V$, for some open neighborhood $V$ 
of the identity element in $H$ ($V$ is taken $k$ times as a factor). Let 
$M=\max\{|m_1|,\ldots,|m_k|\}$ and choose an open symmetric neighborhood 
$W$ of the identity in $H$ such that
$$
\underbrace{W+\cdots+W}_{M\ \rm times}\subset V.
$$
It follows from our choice of $M$ and $W$ that $\lambda(W)\subset U$. Further,
the equality $\varphi\circ\lambda=Id_H$ implies that $W=\varphi(\lambda(W))\subset 
\varphi(U)$, so the set $\varphi(U)$ contains a nonempty open neighborhood of
the identity. Hence the homomorphism $\varphi$ is open. 
\end{proof}

The conclusion of the next lemma is only formally stronger than that of 
Theorem~\ref{Th:Ans}. 

\begin{lemma}\label{Le:Dis}
Let $G\subset \T^\cont$ be the group constructed in Theorem~\ref{Th:Ans} and let $H$ 
be a quotient group of $G$. Then every countable subgroup of $H$ is $h$-embedded
and closed in $H$, while every quotient group of $H$ is either trivial or non-separable.
\end{lemma}

\begin{proof}
Since every quotient group of $H$ is also a quotient of $G$, the second part 
of the conclusion is immediate from Theorem~\ref{Th:Ans}. Therefore, according
to Lemma~\ref{Le:ACDT}, it suffices to verify that every countable subgroup of
$H$ is $h$-embedded. 

If $H$ is trivial, there is nothing to prove. Assume therefore that $|H|>1$. Denote
by $p$ an open continuous homomorphism of $G$ onto $H$ and let $K$ be the
kernel of $p$. Clearly the group $H$ is precompact and, hence, there is a non-trivial
continuous character $\chi$ on $H$. Since the kernel of $p$ is contained in the 
kernel of the continuous character $\chi\circ{p}$ of $G$ and each non-trivial
character of $G$ has countable kernel, we conclude that the kernel of $p$ is
countable. 

Let $C$ be a countable subgroup of $H$ and $f$ a homomorphism of $C$ to $\T$.
Then $D=p^{-1}(C)$ is a countable subgroup of $G$ and $f_D=f\circ p\res D$
is a homomorphism of $D$ to $\T$. Since $D$ is $h$-embedded in $G$, $f_D$
admits an extension to a \emph{continuous} character $f^\ast$ on $G$. It follows 
from $D=p^{-1}(C)$ and the definition of $f_D$ and $f^\ast$ that $\ker{p}\subset 
\ker f_D\subset \ker{f^\ast}$. Therefore, by \cite[Corollary~1.5.11]{AT}, there exists 
a continuous character $g$ of $H$ satisfying $f^\ast=g\circ{p}$. Then $g\circ{p}\res D
=f^\ast\hskip-1.4pt\res D=f_D=f\circ{p}\res D$, whence it follows that $g$ extends
$f$. Hence the group $C=p(D)$ is $h$-embedded in $H$.
\end{proof}

In the following lemma we establish that one of the properties of the group $G$
constructed in Theorem~\ref{Th:Ans} is finitely productive. 

\begin{lemma}\label{Le:FP}
Let $H_1,\ldots,H_n$ be abelian topological groups such that for each $i\leq n$,
all countable subgroups of $H_i$ are $h$-embedded. Then all countable subgroups
of the group $H=H_1\times\cdots\times H_n$ are also $h$-embedded.
\end{lemma}

\begin{proof}
Let $C$ be a countable subgroup of $H$ and $f\colon C\to\T$ a homomorphism. 
For every $i\leq n$, let $C_i=p_i(C)$, where $p_i\colon H\to H_i$ is the projection. 
Then $C$ is a subgroup of the countable group $D=C_1\times\cdots\times C_n$. 
Since the group $\T$ is divisible, $f$ extends to a homomorphism $f_D\colon D
\to\T$. Taking the restrictions of $f_D$ to the factors, we can find homomorphisms
$f_i\colon C_i\to\T$, where $i=1,\ldots,n$ such that 
$$
f_D(x_1,\ldots,x_n)=f_1(x_1)\cdots f_n(x_n)
$$
for each $(x_1,\ldots,x_n)\in C_1\times\cdots\times C_n$. For every $i\leq n$,
there exists a continuous character $g_i$ on $H_i$ extending $f_i$. We define
a continuous character $g$ on $H$ by
$$
g(y_1,\ldots,y_n)=g_1(y_1)\cdots g_n(y_n),
$$
for each $(y_1,\ldots,y_n)\in H_1\times\cdots\times H_n$. Then $g$ extends both
$f_D$ and $f$. So the group $C$ is $h$-embedded in $H$.
\end{proof}

Let $p\colon G\to H$ be a continuous surjective homomorphism of topological
groups and $\tau_H$ be the topology of $H$. Denote by $\tau_H^p$ the finest 
topology on $H$ such that the mapping $p$ of $G$ to $G^\ast=(H,\tau_H^p)$ 
is continuous. It is clear that $G^\ast$ is a topological group, $\tau_H\subset\tau_H^p$, 
and the homomorphism $p\colon G\to G^\ast$ is open. We will say that $\tau_H^p$
is the \emph{$p$-quotient} topology on $H$. It follows that the groups $G^\ast$
and $G/K$ are  isomorphic as topological groups, where $K$ is the kernel of $p$.
\smallskip 

The proof of the following lemma is elementary, so we omit it. 

\begin{lemma}\label{Le:El}
Let $p\colon G\to H$, $q\colon H\to K$ and $r\colon G\to K$ be continuous
surjective homomorphisms of topological groups satisfying $r=q\circ p$. 
Then the $r$-quotient topology on $K$ is finer than the $q$-quotient topology
of $K$. If the homomorphism $p$ is open, then two topologies on $K$ coincide.
\end{lemma}

The next lemma extends the conclusion of Lemma~\ref{Le:Dis} to finite
powers of the group $G$.

\begin{thm}\label{Th:DisG}
Let $G\subset \T^\cont$ be the group constructed in Theorem~\ref{Th:Ans}
and $k\geq 1$ be an integer. Then every quotient group of $G^k$ is either 
trivial or non-separable.
\end{thm}

\begin{proof}
Let $\varphi\colon G^k\to H$ be an open continuous homomorphism onto
a topological group $H$ with $|H|>1$. It is clear that the group $H$ is precompact
and abelian. Hence there exists a non-trivial continuous character $\chi$ on $H$. 
Then $\chi^\ast=\chi\circ\varphi$ is a non-trivial continuous character on the dense
subgroup $G^k$ of $(\T^\cont)^k$. Let the group $H^\ast=\chi^\ast(G^k)$ carry the 
$\chi^\ast$-quotient topology. By Lemma~\ref{Le:El}, $H^\ast$ is a continuous 
homomorphic image of $H$, so it suffices to prove that $H^\ast$ is not separable. 

Denote by $\psi$ an extension of $\chi^\ast$ to a continuous character of $(\T^\cont)^k$. 
Then there exist continuous characters $\psi_1,\ldots,\psi_k$ on $\T^\cont$ such that 
$\psi(x_1,\ldots,x_k)=\prod_{i=1}^k \psi_i(x_i)$, for all $x_1,\ldots,x_k\in\T^\cont$.

It follows from Lemma~\ref{Le:char} that for every $i\in\{1,\ldots,k\}$, one can find 
ordinals $\alpha_{i,1}<\alpha_{i,2}<\cdots<\alpha_{i,n_i}$ in $\cont$ and integers 
$m_{i,1},m_{i,2},\ldots,m_{i,n_i}$ such that 
$$
\psi_i(x)=\prod_{j=1}^{n_i} x(\alpha_{i,j})^{m_{i,j}}, \mbox{ for each } x\in\T^\cont.
$$
Some of the integers $m_{i,1},m_{i,2},\ldots,m_{i,n_i}$ can be zeros. Therefore
we can assume that the sets $A_i=\{\alpha_{i,1},\alpha_{i,2},\ldots,\alpha_{i,n_i}\}$
coincide for all $i\leq k$. Let $A_1=A_2=\cdots=A_k=A$. In particular, $n_i=n_{i'}=n$ 
for all $i,i'\leq k$ and, for each $j\leq n$, there exists an ordinal $\alpha_j\in\cont$ 
such that $\alpha_{1,j}=\alpha_{2,j}=\cdots=\alpha_{k,j}=\alpha_j$; hence 
$A=\{\alpha_1,\ldots,\alpha_n\}$. In other words, the columns of the rectangular 
$k\times n$ matrix $\big(\alpha_{i,j}\big)$ are constant. Hence, combining the 
expressions for the characters $\psi$ and $\psi_i$, $1\leq i\leq k$, and changing 
the order of multiplication, we get the following equality:
\begin{eqnarray}
\psi(x_1,\ldots,x_k)=\prod_{j=1}^n\prod_{i=1}^k x_i(\alpha_{j})^{m_{i,j}}.\label{Eq:4}
\end{eqnarray}
Furthermore, we can additionally assume that the set $A\subset\cont$ is minimal by 
inclusion among those that admit a representation of $\psi$ in the form $(\ref{Eq:4})$. 
Therefore $\sum_{i=1}^k |m_{i,j}|>0$, for each $j\in\{1,\ldots,n\}$.

For every $j\leq n$, denote by $\lambda_j$ the continuous character on $(\T^\cont)^k$ 
defined by 
\begin{eqnarray}	
\lambda_j(x_1,\ldots,x_k)=\prod_{i=1}^k x_i(\alpha_j)^{m_{i,j}}, \label{Eq:4.1}
\end{eqnarray}
where $x_1,\ldots,x_k\in\T^\cont$. It follows from $(\ref{Eq:4})$ and $(\ref{Eq:4.1})$ 
that $\psi=\lambda_1\cdot\ldots\cdot\lambda_n$. For every $j\leq n$, the character 
$\lambda_j$ can be represented in the form $\lambda_j^\ast\circ (p_{\alpha_j})^k$, 
where $p_{\alpha_j}$ is the projection of $\T^\cont$ to the factor $\T_{(\alpha_j)}$ 
and $\lambda^\ast_j$ is the continuous character on $\T_{(\alpha_j)}^k$ defined by
$$
\lambda^\ast_j(t_1,\ldots,t_k)=t_1^{m_{1,j}}\cdot\ldots\cdot t_k^{m_{k,j}}.
$$ 

We supply the subgroup $H_j=p_{\alpha_j}(G)$ of $\T$ with the $q_j$-quotient 
topology, where $q_j=p_{\alpha_j}\hskip-1.5pt\res G$ and $1\leq j\leq n$. Since
the kernel of $q_j$ is countable, the group $H_j$ is uncountable. By 
Lemma~\ref{Le:Dis}, every countable subgroup of $H_j$ is $h$-embedded. 
Let $M_j$ be the maximal common divisor of the integers $m_{1,j},\ldots,m_{k,j}$ 
(we recall that at least one of these integers is distinct from zero). For each $i\leq k$, 
we put $n_{i,j}=m_{i,j}/M_j$ and consider the character $\delta_j$ on $\T^k$ defined by
$$
\delta_j(t_1,\ldots,t_k)=t_1^{n_{1,j}}\cdot \ldots \cdot t_k^{n_{k,j}}.
$$
It is clear from the definition that $\lambda_j^\ast=\delta_j^{M_j}$. Since the maximal
common divisor of $n_{1,j},\ldots,n_{k,j}$ is equal to $1$, it follows from Lemma~\ref{Le:ML} 
that the homomorphism $\delta_j\colon H_j^k\to H_j$ is open and surjective. 
Notice that the group $H_j$ is algebraically a subgroup of $\T$.

Let $v_j\colon \T\to \T$ be the homomorphism defined by $v_j(t)=t^{M_j}$, for each
$t\in \T$. It follows from $\lambda_j^\ast=\delta_j^{M_j}$ that $\lambda_j^\ast=v_j\circ 
\delta_j$. Clearly the kernel of $v_j$ is finite and the restriction $v_j\hskip-1pt\res H_j$ 
is continuous when considered as a homomorphism of $H_j$ to itself. Hence the group
$\lambda_j^\ast(H_j^k)=v_j(H_j)$ is uncountable.
\[
\xymatrix{H_j^k\ar@{>}[r]^{\delta_j} \ar@{>}[d]_{\lambda^\ast_j} & H_j  \ar@{>}[dl]^{v_j}  \\
H_j & }
\]
Strictly speaking, one should replace $\delta_j$, $\lambda_j^\ast$ and $v_j$ in the 
above diagram by $\delta_j\res H_j^k$, $\lambda_j^\ast\res H_j^k$ and $v_j\hskip-1pt\res H_j$, 
respectively.

We define a homomorphism 
$$
\lambda^\ast\colon (\T_{(\alpha_1)})^k\times \cdots
\times (\T_{(\alpha_n)})^k \to \T_{(\alpha_1)}\times\cdots\times \T_{(\alpha_n)}
$$ 
by $\lambda^\ast(y_1,\ldots,y_n)=(\lambda_1^\ast(y_1),\ldots,\lambda_n^\ast(y_n))$. 
In other words, $\lambda^\ast=\lambda_1^\ast\times\cdots\times \lambda_n^\ast$. 
Since each group $H_j$ is algebraically identified with the subgroup $p_{\alpha_j}(G)$ 
of $\T$ we have that $\lambda^\ast(H_1^k\times\cdots\times H_n^k)\subset H_1\times
\cdots\times H_n$. Further, the multiplication mapping $P$ of $\T^n$ to $\T$ defined 
by $P(t_1,\ldots,t_n)=t_1\cdots t_n$ is a continuous character on $\T^n$. It follows 
from the equality 
$$
\psi=\lambda_1\cdot\ldots\cdot\lambda_n=P\circ (\lambda_1^\ast\times\cdots
\times\lambda_n^\ast) \circ (p_{\alpha_1}^k\times\cdots\times p_{\alpha_n}^k)
$$ 
and our definition of $\lambda^\ast$ that $\chi^\ast=\psi\res G^k=
P\circ \lambda^\ast\circ (p_A)^k\res G^k$, where $p_A$ is the projection of 
$\T^\cont$ to $\T^A$. Let $P^\ast$ be the restriction of $P$ to the (abstract) 
subgroup $S=\lambda^\ast(p_A(G)^k)$ of $\T^n$. We consider $S$ with the 
topology inherited from $H_1\times\cdots\times H_n$.
\[
\xymatrix{G^k \ar@{>}[r]^{\chi^\ast} \ar@{>}[d]_{p_A^k} & H^\ast  \ar@{>}[r]^{id} & \T  \\
(p_A(G))^k \ar@{>}[r]^{\,\,\,\,\,\,\,\lambda^\ast}  & S\, \ar@{>}[u]_{P^\ast} \ar@{^{(}->}[r]
& H_1\times\cdots \times H_n \ar@{>}[r]^{\;\;\;\;\;\;\;\;id}   & 
\T^n \ar@{>}[ul]_{P}}
\]
The natural isomorphic embeddings of the abstract groups $H^\ast$ to $\T$ and 
$H_1\times\cdots\times H_n$ to $\T^n$ in the above diagram are denoted by the 
same symbol $id$.

Our next step is to verify the following:
\smallskip\smallskip

\textbf{Claim~3.} \textit{The kernel of the homomorphism $P\colon H_1\times\cdots
\times H_n\to \T$ is countable, where each group $H_j$ is identified algebraically 
with the corresponding subgroup $p_{\alpha_j}(G)$ of $\T$.}
\smallskip

Indeed, take an arbitrary element $(y_1,\ldots,y_n)\in H_1\times\cdots\times H_n$
and assume that $y_1\cdots y_n=1$.  It follows from condition (iv) in the proof of Theorem~\ref{Th:Ans} that for every $j\leq n$, the element $y_j\in H_j=p_{\alpha_j}(G)
=f_{\alpha_j}(A(X))$ can be written in the form $y_j=t_j\cdot z_j$, where 
$t_j\in \hull{f_{\alpha_j}(Y_{\alpha_j})}$ and $z_j\in \hull{E_{\alpha_j}}$. Therefore 
we have the equality $z_1\cdots z_n=t_1^{-1}\cdots t_n^{-1}\in\hull{Y}$, where 
$Y=\bigcup_{j=1}^n f_{\alpha_j}(Y_{\alpha_j})$. Since each set $Y_{\alpha_j}$ 
is countable, the group $\hull{Y}$ is countable as well. Further, since the set 
$E\subset\T$ is independent and the subsets $E_{\alpha_1},\ldots,E_{\alpha_n}$ 
of $E$ are pairwise disjoint, the products $z_1\cdots z_n$ and $z_1^\prime\cdots 
z_n^\prime$ are distinct provided $(z_1,\ldots,z_n)$ and $(z_1^\prime,\ldots,
z_n^\prime)$ are distinct elements of $\hull{E_1}\times\cdots\times\hull{E_n}$. 
Therefore it follows from $|\hull{Y}|\leq\omega$ that there exist at most countably 
many $n$-tuples $(z_1,\ldots,z_n)\in \hull{E_1}\times\cdots\times\hull{E_n}$ with 
$z_1\cdots z_n\in \hull{Y}$. Finally, since the group $\hull{Y}$ is countable, 
we conclude that the kernel of $P$ is countable. This proves Claim~3.

To finish the proof of the theorem we argue as follows. We know that for every
$j\leq n$, all countable subgroups of $H_j$ are $h$-embedded. By Lemma~\ref{Le:FP},
the countable subgroups of $H_1\times\cdots\times H_n$ are also $h$-embedded.
Hence every countable subgroup of $H_1\times\cdots \times H_n$ is closed (see
Lemma~\ref{Le:ACDT}), and the same holds for its subgroup $S$. For every $j\leq n$, 
let $\pi_j$ be the projection of $\T_{(\alpha_1)}\times\cdots\times \T_{(\alpha_n)}$ 
to the factor $\T_{(\alpha_j)}$. It follows from the definition of $\lambda^\ast$ that 
$\pi_j\circ\lambda^\ast = \lambda_j^\ast\circ\pi_j^k$. Therefore 
$\pi_j(S) = \pi_j(\lambda^\ast(p_A(G)^k))  = \lambda_j^\ast(H_j^k)=v_j(H_j)$ 
and the latter group is uncountable. Hence $|S|>\omega$. By Claim~3, the kernel 
of the homomorphism $P$ restricted to $H_1\times\cdots\times H_n$ is countable. 
Therefore the group $H^\ast=P^\ast(S)$ is uncountable as well.

By Lemma~\ref{Le:El}, the $\chi^\ast$-quotient topology of $H^\ast$, i.e.~the original
topology of $H^\ast$ is finer than the $P^\ast$-quotient topology of $H^\ast$ denoted 
by $\tau^\ast$. Since all countable subgroups of $S$ are closed, it follows that the 
group $(H^\ast,\tau^\ast)$ has the same property. Hence all countable subgroups 
of $H^\ast$ are closed as well. Since the group $H^\ast$ is uncountable it cannot 
be separable. 
\end{proof}

We have established in the proof of Theorem~\ref{Th:DisG} that all countable subgroups
of the group $H^\ast=\chi^\ast(G^k)$ are closed provided that $H^\ast$ carries the
$\chi^\ast$-quotient topology. Since $\chi^\ast$ is a continuous character on $G^k$,
this prompts the following question:

\begin{problem}\label{Prob:0}
Let $G$ be the group constructed in Theorem~\ref{Th:Ans} and $H$ a quotient
group of $G^k$, for some integer $k\geq 1$. Are all countable subgroups of $H$ 
$h$-embedded or closed?
\end{problem}

Finally we extend Theorem~\ref{Th:DisG} to arbitrary powers of the group $G$.

\begin{thm}\label{Th:ArP}
Let $G\subset \T^\cont$ be the group constructed in Theorem~\ref{Th:Ans}
and $\tau\geq 1$ be a cardinal. Then every quotient group of $G^\tau$ is either 
trivial or non-separable.
\end{thm}

\begin{proof}
We start as in the proof of Theorem~\ref{Th:DisG}. Let $\varphi\colon G^\tau\to H$ be
an open continuous homomorphism onto a non-trivial Hausdorff topological group $H$. 
Then the group $H$ is precompact and abelian. Hence there exists a continuous character
$\chi$ on $H$ distinct from the trivial one, so $\chi^\ast=\chi\circ\varphi$ is a non-trivial
continuous character on $G^\tau$. We consider the group $H^\ast=\chi^\ast(G^\tau)$
endowed with the $\chi^\ast$-quotient topology. Since $H^\ast$ is a continuous 
homomorphic image of $H$, it suffices to verify that $H^\ast$ is not separable. 

The group $G^\tau$ is a dense subgroup of the compact group $(\T^\cont)^\tau\cong \T^{\cont\times\tau}$. Let $\psi$ be a continuous character on $\T^{\cont\times\tau}$ 
extending $\chi^\ast$. By Lemma~\ref{Le:char}, the character $\psi$ depends on at 
most finitely many coordinates. In other words, we can find finite nonempty sets 
$A\subset \cont$ and $B\subset\tau$ and a continuous character $\psi^\ast$ on the 
group $\T^{A\times B}\cong (\T^A)^B$ such that $\psi=\psi^\ast\circ p_{A,B}$, where 
$p_{A,B}$ is the projection of $\T^{\cont\times\tau}$ onto $\T^{A\times B}$. Let 
the subgroup $K=p_A(G)$ of $\T^A$ carry the quotient topology with respect to
the homomorphism $p_A\res G$, where $p_A\colon \T^\cont\to \T^A$ is the projection. 
Let also $\psi^\ast_B$ be the restriction of $\psi^\ast$ to $K^B=p_{A,B}(G^\tau)$. 
Then $\chi^\ast=\psi^\ast_B \circ p_{A,B}\hskip-1.5pt\res G^\tau$. Further, we can 
represent the restriction $p=p_{A,B}\hskip-1.5pt\res G^\tau$ as the composition of 
the projection $\pi_B\colon G^\tau\to G^B$ and the homomorphism $(p_A)^B\colon 
G^B\to K^B$, which makes the following diagram to commute.
\[
\xymatrix{G^B  \ar@{>}[dr]_{(p_A)^B}   & G^\tau  \ar@{>}[l]_{\pi_B}   
\ar@{>}[dr]^{\chi^\ast} \ar@{>}[d]_{p} &  \\
& K^B \ar@{>}[r]^{\psi_B^\ast}  &   H^\ast  } 
\]

Since the mapping $p$ is open, all homomorphisms in the diagram are continuous.
Clearly $f=\psi_B^\ast\circ (p_A)^B$ is a continuous homomorphism of $G^B$
onto $H^\ast$. Since $\chi^\ast=f\circ\pi_B$, it follows from Lemma~\ref{Le:El}
that the $\chi^\ast$-quotient topology on $H^\ast$ is finer than the $f$-quotient
topology on $H^\ast$, say, $\tau_f$. Further, as the set $B$ is finite, we can apply Theorem~\ref{Th:DisG} to the continuous homomorphism $f\colon G\to H^\ast$ 
and deduce that the group $(H^\ast,\tau_f)$ is not separable. Therefore $H^\ast$
with the $\chi^\ast$-quotient topology is not separable either. This completes the
proof.
\end{proof}

Finally, we solve Problem~\ref{SQGsigma}.

\begin{thm}\label{Th:CUD}
For every cardinal $\tau\geq\cont$, there exists a precompact topological abelian 
group $H$ with the following properties:
\begin{enumerate}
\item[{\rm (a)}] $w(H)=\tau$;
\item[{\rm (b)}] $H=\bigcup_{n\in\omega} H_n$, where $H_0\subset H_1\subset H_2\subset \cdots$ 
                       are proper closed subgroups of $H$;
\item[{\rm (c)}] every quotient group of $H$ is either trivial or non-separable.
\end{enumerate}
\end{thm}

\begin{proof}
Let $\tau$ be a cardinal with $\tau\geq\cont$ and $G$  the precompact abelian
group as in Theorem~\ref{Th:Ans}. Then $G$ is a dense subgroup of $\T^\cont$, so
$w(G)=\cont$. Clearly the group $K=G^\tau$ is precompact, abelian, has weight $\tau$, 
and every non-trivial quotient of $K$ is not separable (see Theorem~\ref{Th:ArP}). 

We denote by $\Pi$ the group $K^\omega$ endowed with the Tychonoff
product topology. For every $n\in\omega$, let $p_n$ be the projection of $\Pi$
to the $n$th factor $K_{(n)}$. Given an element $x\in \Pi$, we denote by
$\supp(x)$ the set $\{n\in\omega: p_n(x)\neq e\}$, where $e$ is the identity
element of $K$. We also put
$$
H=\{x\in \Pi: |\supp(x)|<\omega\}.
$$
It is clear that $H$ is a dense subgroup of $\Pi$ and $w(H)=w(\Pi)=w(K)=\tau$.
For every $n\in\omega$, denote by $H_n$ the subgroup of $H$ which consists
of all $x\in\Pi$ with $\supp(x)\subset\{0,1,\ldots,n\}$. It is easy to see that 
$H_n\cong K^{n+1}$ is closed in $H$, $H_n\subset H_{n+1}$ for each 
$n\in\omega$, and $H=\bigcup_{n\in\omega} H_n$. 

Let us show that every non-trivial quotient of $H$ is not separable. Consider an 
open continuous homomorphism $f\colon H\to L$ onto a non-trivial topological group
$L$. Let $N$ be the kernel of $f$ and $\overline{N}$ the closure of $N$ in $\Pi$. 
Denote by $\pi$ the quotient homomorphism of $\Pi$ onto $\Pi/\overline{N}$. Since 
$N$ is dense in $\overline{N}$, the dense subgroup $\pi(H)$ of $\Pi/\overline{N}$ 
is isomorphic as a topological group to  $L\cong H/N$ (see \cite[Theorem~1.5.16]{AT}). 
The group $\Pi$ is  isomorphic as a topological group   to $(G^\tau)^\omega\cong G^\tau$, so Theorem~\ref{Th:ArP}
 implies that the non-trivial group $\Pi/\overline{N}$ is not 
separable. Thus neither is the group $L\cong H/N$ which is isomorphic as a topological group to the dense subgroup $\pi(H)$ of $\Pi/\overline{N}$.
\end{proof}

\begin{problem}\label{Prob:2}
Does there exist a precompact abelian group $G$ as in Theorem~\ref{Th:Ans} 
which has one of the following additional properties: 
\begin{enumerate}
\item[{\rm (a)}] $G$ is connected;
\item[{\rm (b)}] $G$ is  Baire;
\item[{\rm (c)}] $G$ is reflexive?
\end{enumerate}
\end{problem}

\begin{problem}\label{Prob:3}
Is it true that every precompact reflexive topological group
has the Baire property? 
\end{problem}


\section{Conclusion}\label{Conclusion} 

\noindent Throughout this paper we have listed many problems some of which we have answered completely, others partially and some we have not addressed here. For clarity we state the status of each problem in a table and for easy reference purposes we include a table listing properties. All topological groups in the left column of the first table are assumed to be infinite.

\newpage

$$
\begin{tabular}{|P{3cm}|P{2.0cm}|P{1.8cm}|P{1.8cm}|P{2.0cm}|}
 \hline
{\phantom{junk}\bf Groups} & {\bf Non-trivial Separable Quotient} & {\bf Infinite Separable Quotient} & {\bf Metrizable Separable Quotient} & {\bf Infinite Metrizable Separable Quotient} \\
 \hline
non-totally disconnected & No & No & No & No\\
 \hline
non-discrete $G_{\sigma}$ & No & No & No & No\\
 \hline
reflexive & ? & ? & ? & ?\\
 \hline
 locally compact abelian & Yes & Yes & Yes & Yes\\
 \hline
non-totally disconnected locally compact &\phantom{longstuff} ? & \phantom{longstuff} ? &\phantom{longstuff} ? &\phantom{longstuff} ?\\
 \hline
compact & Yes & Yes & Yes & Yes\\
 \hline
non-discrete $\sigma$-compact & Yes & Yes & No & No\\
 \hline
pseudocompact & Yes & Yes & Yes & Yes\\
 \hline
non-totally disconnected proto-Lie &\phantom{longstuff} Yes &\phantom{longstuff} Yes &\phantom{longstuff} Yes &\phantom{longstuff} Yes\\
 \hline
$\sigma$-compact pro-Lie & Yes & Yes & Yes & Yes\\
 \hline
abelian pro-Lie & Yes & Yes & Yes & Yes\\
 \hline
$\sigma$-compact locally compact & Yes & Yes & Yes & Yes\\
 \hline
almost connected locally compact & Yes & Yes & Yes & Yes\\
 \hline
precompact & No & No & No & No\\
 \hline
$\R$-factorizable & No & No & No & No\\
 \hline
\end{tabular} $$

\newpage
$$\begin{tabular}{|P{2cm}|P{1.6cm}|P{4.8cm}|}
 \hline
{\bf Problem} &{\bf Answer} &{\bf Theorems}  \\
 \hline
\ref{SQProblem} & Partial & \ref{SchauderBasis}, \ref{SaxonWilansky},  \ref{AmirLindenstrauss}, \ref{reflexiveBanach}, \ref{separablesubspaceofdual}, \ref{Argyros1} \\
 \hline
\ref{SQSchauder} & Partial & \ref{JohnsonRosenthal}, \ref{SchauderBasis},  \ref{SaxonWilansky}, \ref{AmirLindenstrauss}, \ref{reflexiveBanach}, \ref{separablesubspaceofdual}, \ref{Argyros1}\\
 \hline
\ref{UnionDenseLinearSubspaces} & Partial & \ref{SaxonWilansky}, \ref{AmirLindenstrauss}, \ref{reflexiveBanach}, \ref{separablesubspaceofdual}, \ref{Argyros1}\\
 \hline
 \ref{barrelled} & Partial& \ref{SaxonWilansky}, \ref{AmirLindenstrauss}, \ref{reflexiveBanach}, \ref{separablesubspaceofdual}, \ref{Argyros1}  \\
 \hline
 \ref{SQlcs} & No & \ref{SQFrechet}, \ref{strict(LF)}, \ref{barrelledlcs}, \ref{C(X)}, \ref{C(X)duals} \\
 \hline
  \ref{SQTopGps} &No& \ref{Th:Ans} \\
 \hline
    \ref{SQInfiniteTopGps} &No & \ref{Th:Ans} \\
 \hline
  \ref{SQMetrizableTopGps} &No & \ref{Th:Ans} \\
 \hline
   \ref{SQInfiniteMetrizableTopGps} &No & \ref{Th:Ans} \\
 \hline
    \ref{SQGsigma} &No & \ref{Th:CUD} \\
 \hline
     \ref{SQReflexiveTopGp} & partial & \ref{SaxonWilansky},  \ref{AmirLindenstrauss}, \ref{reflexiveBanach}, \ref{separablesubspaceofdual}, \ref{Argyros1}, \ref{SQLCATheorem}\\
 \hline
     \ref{SQLCA} &Yes & \ref{SQLCATheorem} \\
 \hline
     \ref{SQBanachTopGp} &partial &  \ref{SaxonWilansky}, \ref{AmirLindenstrauss}, 
     \ref{reflexiveBanach}, \ref{separablesubspaceofdual}, \ref{Argyros1}\\
 \hline
     \ref{SQLC} &partial& \ref{SQTheoremCompact}, \ref{SQLCATheorem}, 
     \ref{SQSigma-CompactLCTheorem}, \ref{SQAlmostConnectedTheorem}, \ref{Cor:sigma-c} \\
 \hline
     \ref{SQCompact} &Yes& \ref{SQTheoremCompact}, \ref{Cor:sigma-c}, \ref{Th:Pseu} \\
 \hline
 \ref{SQsigma-compact}(i)\hskip-0.1pt \&(ii)& Yes&\ref{Cor:sigma-c} \\
 \hline
  \ref{SQsigma-compact}(iii)\hskip-0.1pt \&(iv)& No&\ref{Ex:Qs} \\
 \hline
 \ref{SQprecompact} & No & \ref{Th:Ans}\\
 \hline
 \ref{SQpseudocompact} &Yes &\ref{Th:Pseu}\\
 \hline
 \ref{SQRfactorizable} & No &\ref{coro:Rfac}\\
 \hline
 \ref{Prob:0} & ? &{\bf ---}\\
 \hline
 \ref{Prob:2} & ? &{\bf ---}\\
 \hline
 \ref{Prob:3} & ? &{\bf---}\\
 \hline
 \end{tabular}$$


\end{document}